\documentclass[runningheads]{llncs}
\usepackage[T1]{fontenc}
\usepackage{graphicx}
\usepackage{amsbsy}
\usepackage{amssymb}
\usepackage{amsmath,amsthm,amsfonts}
\usepackage{appendix}
\usepackage{mathtools}
\usepackage{thmtools, thm-restate}
\usepackage{chngcntr}
\usepackage{hyperref}
\usepackage{tikz}
\usetikzlibrary{shapes.geometric, arrows,arrows.meta}
\usetikzlibrary{positioning}
 \newcommand{\indep}{\perp \!\!\! \perp} 
 \newcommand{\defeq}{\coloneqq}
 \allowdisplaybreaks
\begin{document}
\title{Causality for VARMA processes with instantaneous effects: The global Markov property, faithfulness and instrumental variables}
\titlerunning{Causality for VARMA processes with instantaneous effects}
%
\author{Ignacio Gonz\'alez-P\'erez }
\authorrunning{I. Gonz\'alez-P\'erez }
%
\institute{Department of mathematics, \'Ecole Polytechnique F\'ed\'erale de Lausanne \\ 
\email{ignacio.gonzalezperez@epfl.ch}}
\maketitle              
\begin{abstract}
Causal  reasoning has gained great attention over the last half century as it allows (or at least intends) to answer questions which go above those within the capabilities of classical inferential statistics using just observational data. So far, causal research has been focused mostly on the i.i.d.  setting. However, many are the situations where  there exists a non-trivial dependence structure between sequential observations. Motivated by this fact, the main purpose of this work is to study causal properties of time series under the structural assumption of a VARMA model with instantaneous effects. First, the global Markov property is studied, building on existing work for VAR processes without instantaneous effects. Infinite graphs which represent the dependencies of the process are defined so that separation statements translate to conditional independencies in the stationary distribution of the process. Second, faithfulness is examined as a counterpart of this Markov property. Conditions are given so that the stationary distribution of the process is almost surely faithful to said infinite graphs.  In addition, an instrumental variable regression framework is developed for VARMA models with instantaneous effects. This allows to identify and consistently estimate total causal effects. 
\end{abstract}
\textbf{Note: } Most of this work was done while the author was writing his MSc thesis at ETH Zurich.
\newpage
\section{Introduction}\label{sec:Intro}
A wide range of statistical methods rely on the assumption of independent observations. The factorization of the joint distribution opens a wide range of possibilities for statistical modelling and inferential methods. However, plenty are the situations where the observed data cannot be considered independent: A non-trivial correlation structure must be assumed. When the data can be understood as observed under a discrete time structure, we enter the realm of time series. A comprehensive introduction to time series modelling can be found in \cite{brockwell1991time}. One of the simplest yet versatile parametric models for time series are VARMA models. Aiming to model the expectation of the process conditional on the past, they assume an additive structure of a finite number of past observations and innovations.  Conditions over the coefficients to achieve properties of these processes, such as stationarity and invertibility, can be found in \cite{hamilton2020time}, together with methods for fitting such parametric models to data.

However, classical statistical inference methods might not be able to answer some scientific questions of interest. For example, when we want to model how a response variable $Y$ reacts to an intervention on $X$, observational statements are not enough: Causal knowledge is required. The field of  causality is a relatively new one, with seminal work usually being attributed to \cite{pearl2000causality} and \cite{spirtes2001causation}. Since then, the field has branched in multiple directions. Causal discovery methods rely on different assumptions to try and recover the true causal structure (or at least its Markov equivalence class). Classical regression methods are in some contexts used to estimate quantities which have  well-defined causal interpretations. These are usually referred to as causal effects.    An introduction to the field of causality can be found in \cite{peters2017elements}. Applicability of causal thinking has been found in multiple fields outside of statistics, such as medicine \cite{rizzi1992causality}, ecology \cite{kimmel2021causal}, psychology \cite{bolger2019causal}, etc.

Most of the existing causal methods rely on the assumption of i.i.d. data. For time series, the most wide-spread framework of causal reasoning is Granger causality, an overview of which can be found in \cite{shojaie2022granger} or in \cite[Sec. 10.3.3]{peters2017elements}. However,  causality for time series outside  the Granger formalism is quite primitive at the time of writing. 
 In \cite{thams2022identifying}  the global Markov property is proven for VAR processes without instantaneous effects, and develop an instrumental variable regression formalism to identify causal effects.  A causal discovery algorithm is developed  
 in \cite{hyvarinen2010estimation}  for VAR processes with instantaneous effects based on the non-Gaussianity assumptions of \cite{shimizu2006linear}. A method to estimate causal effects in non-linear time series and differentiate whether they are instantaneous or lagged was presented in \cite{runge2020discovering}. Contemporaneous to the writing of this work, \cite{hochsprungglobal}  proposed a proof of the global Markov property for a wide class of time series under the assumptions of strict stationarity and mixing (among others). The main purpose of this work is to study causal properties of time series, mostly under the linearity and additive noise assumption of a VARMA model with instantaneous effects. Not only will we study Markov properties, but also dive into causal inference and causal discovery, proposing applicable algorithms whenever possible. 

Therefore, this work is organized as follows: Section \ref{sec:Preliminaries} is dedicated to introduce definitions, graph terminology, auxiliary results, and most notably the global Markov property for VAR processes without instantaneous effects. In Section \ref{sec:GMPVar}  we  study the global Markov property for VAR processes with instantaneous effects, as an extension to \cite[Thm. 1]{thams2022identifying}. The family of processes considered is expanded to VARMA models in Section \ref{sec:GMPVarma}, which is again dedicated to the global Markov property. As  a counterpart of the Markov property, we study in Section \ref{chap:faithfulness} under which conditions the stationary distribution of a VARMA model is faithful to the full-time graph of the process. With these tools at hand, causal inference techniques  are developed in Section \ref{chap:IV}. An instrumental variable regression framework is designed to identify and consistently estimate total causal effects for VARMA processes with instantaneous effects. Finally, a discussion of the achievements, shortcomings and possible extensions of this work is presented in Section \ref{sec:Conclusion}. 
\section{Preliminaries}\label{sec:Preliminaries}
This section begins by introducing graph terminology and separation characterizations for directed acyclic graphs and acyclic directed mixed graphs. Most of the proofs of  the results presented in this chapter can be found in the given references. However, we have stated some results which we formulated to cater  the specific needs of the theory developed in this work. The proofs of these results can be found in Appendix \ref{sec:AppendixProofs}. Therefore, a reader with some previous experience in the field of causality might consider skipping Sections \ref{sec:PrelimDAG} and \ref{sec:PrelimADMG}. Following, in Section \ref{sec:GmpVARWithout} we introduce the global Markov property for VAR processes without instantaneous effects \cite[Thm. 1]{thams2022identifying}, which will then be extended in Section \ref{sec:GMP}.
\subsection{Directed acyclic graphs}\label{sec:PrelimDAG}
The first type of graphs we will consider are directed acyclic graphs. These graphs have only directed edges, and have no cycles.
\begin{definition}[Directed acyclic graph \texorpdfstring{\cite[Sec. 2.1.1]{lauritzen1996graphical}})] 
    A directed graph is a pair $\mathcal{G}=(V,E)$, with $V$ a set of nodes and $E\subset V\times V$. We represent this graph by drawing directed edges $v_i\to v_j$ if $(v_i,v_j)\in E$. A path in $\mathcal{G}$ is an alternating sequence $\mathfrak{p} =(v_1,e_1,v_2,\ldots,e_{n-1},v_n)$ of distinct vertices $v_i$ and edges $e_i$ such that $v_i$ and $v_{i+1}$ are connected by $e_i$. We say that $\mathfrak{p}$ is a directed path from $v_1$ to $v_n$ if from every $i$, $e_i$ points from $v_i$ to $v_{i+1}$. We say that a directed graph is acyclic (and we call it a directed acyclic graph (DAG)) if there is no directed path between two nodes $v_1$ to $v_n$ such that also $(v_n,v_1)\in E$. 
\end{definition}
\begin{definition}[Descendants, children, ancestors and parents of a node in a DAG \texorpdfstring{\cite[Sec. 2.1.1]{lauritzen1996graphical}})]\label{def:DescendantsDAG}
    Let $\mathcal{G}$ be a DAG over nodes $V$. We say that $v_j$ is a descendant of $v_i$ if there is a directed path in $\mathcal{G}$ from $v_i$ to $v_j$. We denote the set of all descendants of $v_i$ as $DE_{\mathcal{G}}(v_i)\subset V$. In this context we also say $v_i$ is an ancestor of $v_j$. We denote the set of ancestors of  $v\in V$ by $AN_{\mathcal{G}}(v)$. By convention, we assume that a node is an ancestor and a descendant of itself. Furthermore, if $(v_i,v_j)\in E$, we say that $v_i$ is a parent of $v_j$ and that $v_j$ is a child of $v_i$. We denote the sent of parents of $v\in V$ as $PA_{\mathcal{G}}(v)$.
\end{definition}
\begin{definition}[Collider in a path of a DAG \texorpdfstring{\cite[Sec. 2.1.1]{lauritzen1996graphical}})]
    Let $\mathcal{G}$ be a DAG over nodes $V$, and $\mathfrak{p} =(v_1,e_1,v_2,\ldots,e_{n-1},v_n)$ a path in $\mathcal{G}$. We say that $v_i$ is a collider in $\mathfrak{p}$ if $(e_{i-1},v_i,e_i)$ is of the form $\to v_i \leftarrow $. 
\end{definition}
One of the  properties in a directed acyclic graph that we will be interested in is $d$-separation, which speaks about connection in a directed graph.
\begin{definition}[$d$-separation in a DAG \texorpdfstring{\cite[Sec. 2.1.1]{lauritzen1996graphical}})]
    Let $\mathcal{G}$ be a DAG over nodes $V$. Given two distinct nodes $v,w\in V$ and $B\subset V\backslash\{v,w\}$, we say that a path $\mathfrak{p}$ between $v$ and $w$ is blocked by $B$ if one of the following holds:
    \begin{enumerate}
        \item There is a non-collider in $\mathfrak{p}$ which is  in $B$,
        \item there is a collider $v$ in the path such that $DE_{\mathcal{G}}(v)\cap B=\emptyset$.
    \end{enumerate}
    Given $A,B,C\subset V$ finite pairwise disjoint sets of nodes, we say that $A$ and $C$ are $d$-separated by $B$ in $\mathcal{G}$  (and we denote it by $A\perp_{\mathcal{G}} C|B$) if every path from a node in $A$ to a node in $C$ is blocked by $B$.
\end{definition}
A characterization of the $d$-separation statements on a DAG was given by \cite{lauritzen1990independence} based on a translation of the DAG to an undirected graph called the moralized graph of a DAG.
\begin{definition}[Moralized graph of a DAG \texorpdfstring{\cite[Sec. 2.1.1]{lauritzen1996graphical}})]\label{def:MoralizedDAG}
    Given a DAG $\mathcal{G}$, we construct its moralized graph, and denote it by $\mathcal{G}^m$, by drawing an edge between every pair of nodes which have a common child, and then removing the direction of all edges.  
\end{definition}
\begin{proposition}[\cite{lauritzen1990independence}]\label{prop:SeparationMoralized}
    Let $A,B,C$ be pairwise disjoint sets of nodes of a DAG $\mathcal{G}$. Then $A$ and $C$ are $d$-separated by $B$ in $\mathcal{G}$ if and only if $A$ and $C$ are separated by $B$ in $(\mathcal{G}_{AN_{\mathcal{G}}(A\cup B \cup C)})^m$. By $\mathcal{G}_W$ we denote the proper sub-DAG of $\mathcal{G}$ over nodes $W$. 
\end{proposition}
\begin{restatable}{proposition}{propExtension}\label{prop:Extension}
    Consider a DAG $\mathcal{G}$ over nodes $V$. Let $A,B,C\subset V$ pairwise disjoint sets of nodes such that $V=AN_{\mathcal{G}}(A\cup B\cup C)$ and $A\perp_{\mathcal{G}}C|B$. Then $A$ and $C$ can be extended to disjoint sets of nodes which fill $V$ and remain $d$-separated given $B$. More precisely, there exists $A^+,C^+\subset V$ such that $A\subseteq A^+$, $C\subseteq C^+$, $A^+\cap C^+=\emptyset$, $A^+\perp_{\mathcal{G}}C^+|B$ and $V=A^+\cup B \cup C^+$. 
\end{restatable}
\begin{restatable}{corollary}{coroExtension}\label{coro:Extension}
    In the context of Proposition \ref{prop:Extension}, if $(\mathcal{G}^m)_{V\backslash B}$ has only two connected components, then the extension given in said proposition is unique, and  given by those two connected components. 
\end{restatable}
\subsection{Acyclic directed mixed graphs}  \label{sec:PrelimADMG}
DAGs will not be the only kind of graphs which we will be interested in. Directed edges are used to represent dependencies or effects with a clear orientation. However, in some scenarios we want a graph to reflect a non-directional dependence. For this we will use bi-directed edges, what gives rise to directed mixed graphs.
\begin{definition}[Acyclic directed mixed graph \texorpdfstring{\cite[Sec. 1]{richardson2003markov}})]
    A  directed mixed  graph $\mathcal{G}$ over nodes $V$ is a graph over these nodes where both directed $\to$ and bi-directed $\leftrightarrow$ edges are allowed. A path in $\mathcal{G}$ is an alternating sequence $\mathfrak{p} =(v_1,e_1,v_2,\ldots,e_{n-1},v_n)$ of distinct vertices $v_i$ and edges $e_i$ such that $v_i$ and $v_{i+1}$ are connected by $e_i$. We say that $\mathfrak{p}$ is a directed path from $v_1$ to $v_n$ if for every $i$, $e_i$ is a directed edge  pointing from $v_i$ to $v_{i+1}$. We say that a  directed mixed graph is acyclic (and we call it an acyclic directed mixed  graph (ADMG)) if there is no directed path between two nodes $v_1$ to $v_n$ such that the edge $v_n\to v_1$ also exists. 
\end{definition}
\begin{remark}
    In an ADMG there might be two edges between two nodes, but at least one of them must the bi-directed, as otherwise we would have a cycle. Furthermore, note that a bi-directed edge $v\leftrightarrow w$ does {not} aim to replace $v \to w$ and $w\to v$.
\end{remark}
\begin{definition}[Descendants, children, ancestors, parents and spouses of a node in an ADMG \texorpdfstring{\cite[Sec. 2]{richardson2003markov}})]\label{def:DescendantsADMG}
    Let $\mathcal{G}$ be an ADMG over nodes $V$. Given a node $v\in V$, we define and denote the set of its parents, ancestors, descendants and children as in Definition \ref{def:DescendantsDAG}, meaning with respect only to directed edges. We define the set of spouses of $v$ in $\mathcal{G}$, and denote it as $SP_{\mathcal{G}}(v)$, as the set $\{w\in V\colon v\leftrightarrow w \text{ exists in } \mathcal{G}\}$. By convention we assume a node to be a spouse of itself.
\end{definition}
\begin{definition}[Collider in a path of an ADMG \texorpdfstring{\cite[Sec. 2.1]{richardson2003markov}})]
    Let $\mathcal{G}$ be an ADMG over nodes $V$. Consider a path $\mathfrak{p} =(v_1,e_1,v_2,\ldots,e_{n-1},v_n)$ in $\mathcal{G}$. We say that $v_i$ is a collider in the path if the edges preceding and succeeding $v_i$ in the path have an arrowhead at $v_i$, i.e. we have one of the structures in the path: $\to v_i \leftarrow$, $\to v_i \leftrightarrow$, $\leftrightarrow v_i \leftarrow$, $\leftrightarrow v_i \leftrightarrow$. Otherwise $v_i$ is called a non-collider in the path.
\end{definition}
\begin{definition}[$m$-separation in an ADMG  \texorpdfstring{\cite[Sec. 2.1]{richardson2003markov}})]\label{def:mSep}
    Let $\mathcal{G}$ be an ADMG over nodes $V$. Given two distinct nodes $v,w\in V$ and $B\subset V\backslash\{v,w\}$, we say that a path $\mathfrak{p}$ between $v$ and $w$ is blocked by $B$ if one of the following holds:
    \begin{enumerate}
        \item There is a non-collider in $\mathfrak{p}$ which is  in $B$,
        \item there is a collider $v$ in the path such that $DE_{\mathcal{G}}(v)\cap B=\emptyset$.
    \end{enumerate}
    Given $A,B,C\subset V$ finite pairwise pairwise disjoint sets of nodes, we say that $A$ and $C$ are $m$-separated by $B$ in $\mathcal{G}$  (and we denote it by $A\perp_{\mathcal{G}} C|B$) if every path from a node in $A$ to a node in $C$ is blocked by $B$.
\end{definition}
\begin{remark}
    The concept of $m$-separation in an ADMG is a generalization of  $d$-separation in a DAG. The difference is that a node can be a collider in a path in an ADMG in more ways than in a DAG, due to the presence of bi-directed edges. Furthermore, adjacent colliders in a path in an ADMG can exist.
\end{remark}
As we did with DAGs, we would be interested in translating the $m$-separation statement in an ADMG to an undirected graph. For DAGs we had the moralized graph (Definition \ref{def:MoralizedDAG}). For ADMGs we will have an augmented graph.
\begin{definition}[Augmented graph of an ADMG \texorpdfstring{\cite[Sec. 2.2]{richardson2003markov}})]
    Let $\mathcal{G}$ be an ADMG over nodes $V$. Two nodes $v,w\in V$ are collider connected in $\mathcal{G}$ if there is a path between them such that all intermediate nodes are colliders in the path. By convention  two adjacent nodes are collider connected. The augmented graph of $\mathcal{G}$, denoted by $\mathcal{G}^{aug}$, is an undirected graph over nodes $V$, such that $v-w$ exists in $\mathcal{G}^{aug}$ if and only if $v$ and $w$ are collider connected in $\mathcal{G}$.
\end{definition}
\begin{remark}
    The augmentation of an ADMG is a generalization of the moralization process of a DAG. Indeed,  two nodes of a DAG have a common child if and only if then they are collider connected. 
\end{remark}
\begin{restatable}{proposition}{propSeparationADMG}\label{prop:SeparationADMG}
    Let $A,B,C$ be pairwise  disjoint sets of nodes of an ADMG $\mathcal{G}$. Then $A$ and $C$ are $m$-separated by $B$ in $\mathcal{G}$ if and only if $B$ separates $A$ and $C$ in $(\mathcal{G}_{AN_{\mathcal{G}}(A\cup B\cup C)})^{aug}$
\end{restatable}

\begin{restatable}{proposition}{propExtensionADMG}\label{prop:ExtensionADMG}
    Consider an ADMG $\mathcal{G}$ over nodes $V$. Let $A,B,C\subset V$ pairwise disjoint sets of nodes such that $V=AN_{\mathcal{G}}(A\cup B\cup C)$ and $A\perp_{\mathcal{G}}C|B$. Then $A$ and $C$ can be extended to disjoint sets of nodes which fill $V$ and remain $m$-separated given $B$. More precisely, there exist $A^+,C^+\subset V$ such that $A\subseteq A^+$, $C\subseteq C^+$, $A^+\cap C^+=\emptyset$, $A^+\perp_{\mathcal{G}}C^+|B$ and $V=A^+\cup B \cup C^+$.
\end{restatable}

In Section \ref{sec:GMPVarma}  we will need to restrict a DAG to a certain subset of variables, while still representing the whole set of dependencies represented in the original DAG. To do so, we will use latent projections of DAGs into ADMGs.

\begin{definition}[Latent projection of a DAG \cite{richardson2023nested}]\label{def:LatentProjection}
    Let $\mathcal{G}$ be a DAG over a finite set of nodes $V\cup L$ disjoint union. The latent projection of $\mathcal{G}$ over $V$, denoted a $\mathcal{G}_V$, is an ADMG over $V$ such that for all $v,w\in V$ there is:
    \begin{itemize}
        \item A directed edge $v\to w$ if and only if $v\to w$ exists in $\mathcal{G}$ or  there exist $m\in\mathbb{N}$, $v_1,\ldots,v_m\not\in V$, and a directed path $v\to v_1\to\cdots\to v_m\to w$ in $\mathcal{G}$, 
        \item a bi-directed edge $v\leftrightarrow w$ if and only if there exist $m,n\in\mathbb{N}$, $v_1,\ldots,v_n,$ $w_1,\ldots,w_m,U\not\in V$ such that there exist directed paths $U\to v_1\to\cdots\to v_n\to v$ and $U\to w_1\to\cdots\to w_m\to w$ in $\mathcal{G}$.
    \end{itemize}
\end{definition}
\begin{remark}
    All directed edges between nodes in $V$ are preserved by the latent projection operation. Therefore, the sets of parents, ancestors, children and  descendants  of nodes of $V$ in $V$ do not change in the latent projection, as these concepts involve only directed edges. 
\end{remark}
As mentioned before, the purpose of  latent projections is to find a  graph over $V$ which conveys the same separation information as the original DAG. The following result shows that this is indeed achieved.
\begin{proposition}[Separation equivalence under latent projection \cite{richardson2023nested}]\label{prop:SeparationEquivLatentProjection}
    Let $\mathcal{G}$ be a DAG over nodes $V\cup L$ disjoint union and $\mathcal{G}_V$ its latent projection. Let $A,B,C\subset V$ pairwise disjoint sets of nodes. Then $A$ and $C$ are $d$-separated by $B$ in $\mathcal{G}$ if and only if they are $m$-separated by $B$ in $\mathcal{G}_V$.
\end{proposition}
\subsection{The global Markov property for VAR processes without instantaneous effects  }  \label{sec:GmpVARWithout}
As mentioned in Section \ref{sec:Intro}, causality for time series has been mostly focused around the idea of Granger causality. Outside Granger's framework, the study of causality for time series is still a growing field. \cite{thams2022identifying} adopts a linearity and additive noise assumption, which results on a family of time series called vector auto-regressive processes.
\begin{definition}[VAR process without instantaneous effects \texorpdfstring{\cite[Ch. 10]{hamilton2020time}})]\label{def:VARWithout}
    We say that a  time series $S$ follows a VAR($p$) process without instantaneous effects if there exists $p\in\mathbb{N}$ and coefficient matrices $A_1,\ldots,A_p\in\mathbb{R}^{d\times d}$ such that for all $t\in\mathbb{Z}$: $$S_t=A_1S_{t-1}+\ldots +A_pS_{t-p}+\varepsilon_t,$$ where $A_1,\ldots,A_p$ are such that $\det (I_d\lambda^p-A_1\lambda^{p-1}-\ldots -A_p)=0$ implies $|\lambda|<1$, and  $\varepsilon_t$ is an i.i.d. process with finite second moments and jointly independent components. We define the full-time graph of this process as an infinite directed graph over nodes $\{S_t^i\colon i\in\{1,\ldots,d\},\ t\in\mathbb{Z}\}$ such that for $t\in\mathbb{Z}$ and $k\in\{1,\ldots,p\}$ there is an edge from $S_{t-k}^j$ to $S_t^i$ if $(A_k)_{i,j}\neq 0$. Note how this graph is always acyclic.
\end{definition}
The first considerable  contribution presented in \cite{thams2022identifying}, namely their Theorem 1, is the global Markov property for VAR processes without instantaneous effects, as stated in the following theorem.
\begin{restatable}[Global Markov property for VAR processes without instantaneous effects \texorpdfstring{\cite[Thm. 1]{thams2022identifying}})]{theorem}{thmGMPWithout} \label{thm:GMPWithout}
    Consider $p\in\mathbb{N}$, and S a VAR($p$) process as defined in \ref{def:VARWithout}. Consider finite pairwise disjoint sets $A,B,C$ of nodes of the full-time graph $\mathcal{G}_{full}$.  If $A$ and $C$ are d-separated by $B$ in $\mathcal{G}_{full}$, then  $A\indep C|B$.
\end{restatable}
\section{The global Markov property  }\label{sec:GMP}

Theorem \ref{thm:GMPWithout}  shows that the global Markov property holds for VAR processes without instantaneous effects with respect to the full-time DAG. First we will study the global Markov property for a slightly larger family of VAR processes: Those with instantaneous effects.
\subsection{The global Markov property for VAR processes with instantaneous effects}\label{sec:GMPVar}
Considering time series models with instantaneous effects is of interest as it can make the Markov equivalence class of a certain full-time DAG non-trivial. Indeed, if two full-time graphs of a VAR process without instantaneous effects are Markov equivalent, then they are equal \cite[Thm. 10.1]{peters2017elements}. This is essentially due to the fact that all edges point forward in time. However, instantaneous effects may create V-structures which might broaden the Markov equivalence class. Therefore, studying VAR processes with instantaneous effects is a problem of interest.
\begin{definition}[VAR process with instantaneous effects]\label{def:VARWith}
    We say that a  time series $S$ follows a VAR($p$) process  with instantaneous effects if there exists $p\in\mathbb{N}$ and coefficient matrices $A_0,A_1,\ldots,A_p\in\mathbb{R}^{d\times d}$ such that ${diag}(A_0)=0$ and for all $t\in\mathbb{Z}$: $$S_t=A_0S_t+A_1S_{t-1}+\ldots +A_pS_{t-p}+\varepsilon_t,$$ where $A_0,A_1,\ldots,A_p$ are such that $\det ((I_d-A_0)\lambda^p-A_1\lambda^{p-1}-\ldots -A_p)=0$ implies $|\lambda|<1$, and  $\varepsilon_t$ is an i.i.d. process with finite second moments and jointly independent components. We define the full-time graph of this process as an infinite directed graph over endogenous nodes $\{S_t^i\colon i\in\{1,\ldots,d\},\ t\in\mathbb{Z}\}$ such that for $t\in\mathbb{Z}$ and $k\in\{0,\ldots,p\}$ there is an edge from $S_{t-k}^j$ to $S_t^i$ if $(A_k)_{i,j}\neq 0$. We assume this graph to be acyclic, for which it is sufficient that the graph of instantaneous edges is acyclic. 
\end{definition}
\begin{remark}
    A VAR($p$) process with instantaneous effects as per Definition \ref{def:VARWith} is distributionally equivalent to a VAR($p$) process without instantaneous effects and correlated contemporaneous innovations. Indeed, as we do not allow for instantaneous cycles, we can reorder the components of $S$ according to the topological ordering of the instantaneous DAG, so that all instantaneous edges ``point downwards'', meaning $A_0$ can be assumed to be triangular inferior with zeroes along the diagonal. Thus,  $I_d-A_0$ is invertible, what allows us to write: 
    \begin{equation}\label{eq:VAREquiv}
        S_t=(I_d-A_0)^{-1}A_1S_{t-1}+\ldots +(I_d-A_0)^{-1}A_p+\delta_t,
    \end{equation}
    with $\delta_t=(I_d-A_0)^{-1}\varepsilon_t$ an i.i.d. process  with covariance matrix $(I_d-A_0)^{-1}\Gamma(I_d-A_0)^{-\top}$ (being $\Gamma=\mathbb{V}\text{ar}(\varepsilon_t)$)\footnote{ $A^{-\top}$ denotes the transpose of the inverse of an invertible matrix $A$.}, meaning $\delta_t$ are potentially correlated innovations. We define the full-time DAG associated with this representation as in Definition \ref{def:VARWithout}. Furthermore, this process without instantaneous effects admits an MA($\infty$) representation:
    \begin{equation*}
        \begin{split}
            0&=\det(I_d\lambda^p-(I_d-A_0)^{-1}A_1\lambda^{p-1}+\ldots +(I_d-A_0)^{-1}A_p)\\&=\det((I_d-A_0)^{-1})\det ((I_d-A_0)\lambda^p-A_1\lambda^{p-1}-\ldots -A_p)\\&=\det ((I_d-A_0)\lambda^p-A_1\lambda^{p-1}-\ldots -A_p)\\&\Rightarrow|\lambda|<1,
        \end{split}
    \end{equation*}
    as per the conditions given in Definition \ref{def:VARWith}. 
\end{remark}
Our first goal   is to generalize Theorem \ref{thm:GMPWithout} to the context of VAR processes with instantaneous effects. To do so, we need to study how the full-time graph changes when rewriting the process as in Equation \eqref{eq:VAREquiv}, and the correlation structure of this process's innovations. 
\begin{restatable}{proposition}{propExtendedDAG}\label{prop:ExtendedDAG}
    Consider a VAR($p$) process with instantaneous effects as per Definition \ref{def:VARWith}, and let $\mathcal{G}_I$ be its associated full-time DAG. Consider the distributionally equivalent VAR($p$) process without instantaneous effects as defined in Equation \eqref{eq:VAREquiv}, and let $\mathcal{G}_{II}$ be its associated full-time DAG. Then $\mathcal{G}_{II}$ contains the same nodes and no more edges than $\mathcal{G}^*$, which is the infinite DAG constructed from $\mathcal{G}_{I}$ by:
    \begin{enumerate}
        \item For each $i\in\{1,\ldots,d\}$ and  $t\in\mathbb{Z}$ draw  an edge from every node in the node set  $PA_{\mathcal{G}_I}(S_t^i)_{[t-p,t-1]}$ to every node in $DE_{\mathcal{G}_I}(S_t^i)_{[t]}$ (instantaneous descendants) if they did not already exist in $\mathcal{G}_I$. This means that we draw into every node edges coming from the non-contemporaneous parents of all its instantaneous ancestors (if they did not exist already).
        \item Remove all instantaneous edges from $\mathcal{G}_I$.
    \end{enumerate}
\end{restatable}
\begin{restatable}{corollary}{coroAncestors}\label{coro:Ancestors}
    In the context of Proposition \ref{prop:ExtendedDAG}, $AN_{\mathcal{G}_{II}}(S_t^i)\subseteq AN_{\mathcal{G}_{I}}(S_t^i)$ for all $t\in\mathbb{Z}$ and $i\in\{1,\ldots,d\}$.
\end{restatable}
An idea to generalize Theorem \ref{thm:GMPWithout} to the scenario with instantaneous effects could be to show that $d$-separation statements translate from $\mathcal{G}_I$ to $\mathcal{G}_{II}$, and then use said theorem in the context of correlated innovations. Unfortunately, $d$-separation statements in  $\mathcal{G}_I$ do not necessarily hold in  $\mathcal{G}_{II}$, as we can see in Example \ref{example:DsepWith}.
\begin{example}\label{example:DsepWith}
    Consider  the 2-dimensional VAR(1) process with instantaneous effects: 
    \begin{equation}\label{eq:ExampleDsepWith}
        \begin{cases}
            X_t=\frac{1}{2}X_{t-1}+\varepsilon_t^X\\
            Y_t=\frac{1}{3}X_{t}+\frac{1}{2}Y_{t-1}+\varepsilon_t^Y
        \end{cases},
    \end{equation}
    with $\varepsilon_t\overset{\text{i.i.d.}}{\sim}\mathcal{N}(0,I_2)$. Its associated full-time DAG $\mathcal{G}_I$ can be seen in Figure \ref{fig:ExampleDsepWith}.     The distribution-equivalent VAR(1) process without instantaneous effects is:
    \begin{equation}\label{eq:ExampleDsepWithout}
        \begin{cases}
            X_t=\frac{1}{2}X_{t-1}+\delta_t^X\\
            Y_t=\frac{1}{6}X_{t-1}+\frac{1}{2}Y_{t-1}+\delta_t^Y
        \end{cases}\text{ with } \delta_t\overset{\text{i.i.d.}}{\sim}\mathcal{N}\left(\begin{pmatrix}
            0\\0\end{pmatrix},\frac{1}{9}\begin{pmatrix}9 & 3 \\ 3 & 10   \end{pmatrix}\right),
    \end{equation}
    and its correspondent full-time DAG $\mathcal{G}_{II}$ can be found in Figure \ref{fig:ExampleDsepWithout}.
    \begin{figure}[ht]
        \centering
        \begin{tikzpicture}[
        main/.style={draw, circle, thick, minimum size=12mm}
        ]
        \node[main] (xt) {$X_t$};
        \node[main] (yt) [below = of xt] {$Y_t$};
        \node[main] (xt-1) [left =  of xt]{$X_{t-1}$};
        \node[main] (yt-1) [below = of xt-1] {$Y_{t-1}$};
        \node[main] (xt-2) [left =  of xt-1]{$X_{t-2}$};
        \node[main] (yt-2) [below = of xt-2] {$Y_{t-2}$};
        
        \node[main] (xt+1) [right =  of xt]{$X_{t+1}$};
        \node[main] (yt+1) [below = of xt+1] {$Y_{t+1}$};
        \node[main, draw=none] (top) [right = of xt+1] {};
        \node[main, draw=none] (bot) [below = of top] {};
        \node[main, draw=none] (pasttop) [left = of xt-2] {};
        \node[main, draw=none] (pastbot) [below = of pasttop] {};
        \draw[-{Stealth[length=2.5mm]}] (xt.south) -- node[right] {$\frac{1}{3}$} (yt.north);
        \draw[-{Stealth[length=2.5mm]}] (xt-1.east) -- node[above] {$\frac{1}{2}$} (xt.west);
        \draw[-{Stealth[length=2.5mm]}] (yt-1.east) -- node[above] {$\frac{1}{2}$} (yt.west);
        \draw[-{Stealth[length=2.5mm]}] (xt-1.south) -- (yt-1.north);
        \draw[-{Stealth[length=2.5mm]}] (xt.east) -- (xt+1.west);
        \draw[-{Stealth[length=2.5mm]}] (yt.east) -- (yt+1.west);
        \draw[-{Stealth[length=2.5mm]}] (xt-2.east) -- (xt-1.west);
        \draw[-{Stealth[length=2.5mm]}] (yt-2.east) -- (yt-1.west);
        \draw[-{Stealth[length=2.5mm]}] (xt-2.south) -- (yt-2.north);
        \draw[-{Stealth[length=2.5mm]}] (xt+1.south) --  (yt+1.north);
        \draw[-{Stealth[length=2.5mm]}] (xt+1.east) -- (top.west);.
        \draw[-{Stealth[length=2.5mm]}] (yt+1.east) -- (bot.west);
        \draw[-{Stealth[length=2.5mm]}] (pasttop.east) -- (xt-2.west);
        \draw[-{Stealth[length=2.5mm]}] (pastbot.east) -- (yt-2.west);
    \end{tikzpicture}
    \caption[Full-time DAG  of the process with instantaneous effects defined in  \eqref{eq:ExampleDsepWith}.]{Full-time DAG $\mathcal{G}_I$ of the time series with instantaneous effects defined in Equation \eqref{eq:ExampleDsepWith}. The coefficients of the edges correspond to the coefficients of the process.}
    \label{fig:ExampleDsepWith}
    \end{figure}
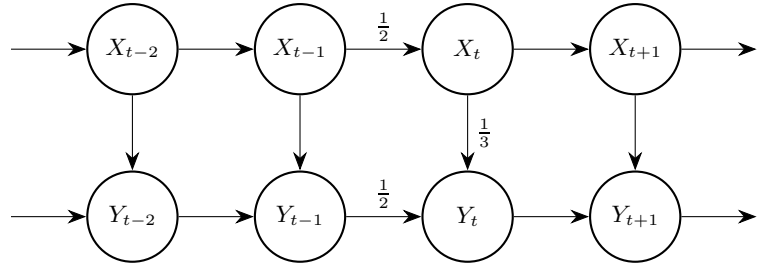
\begin{figure}[ht]
        \centering
        \begin{tikzpicture}[
        main/.style={draw, circle, thick, minimum size=12mm}
        ]
        \node[main] (xt) {$X_t$};
        \node[main] (yt) [below = of xt] {$Y_t$};
        \node[main] (xt-1) [left =  of xt]{$X_{t-1}$};
        \node[main] (yt-1) [below = of xt-1] {$Y_{t-1}$};
        \node[main] (xt-2) [left =  of xt-1]{$X_{t-2}$};
        \node[main] (yt-2) [below = of xt-2] {$Y_{t-2}$};
        
        \node[main] (xt+1) [right =  of xt]{$X_{t+1}$};
        \node[main] (yt+1) [below = of xt+1] {$Y_{t+1}$};
        \node[main, draw=none] (top) [right = of xt+1] {};
        \node[main, draw=none] (bot) [below = of top] {};
        \node[main, draw=none] (pasttop) [left = of xt-2] {};
        \node[main, draw=none] (pastbot) [below = of pasttop] {};
        \draw[-{Stealth[length=2.5mm]}] (xt.south) --  (yt+1.north);
        \draw[-{Stealth[length=2.5mm]}] (xt-1.east) -- node[above] {$\frac{1}{2}$} (xt.west);
        \draw[-{Stealth[length=2.5mm]}] (yt-1.east) -- node[above] {$\frac{1}{2}$} (yt.west);
        \draw[-{Stealth[length=2.5mm]}] (xt-1.south) -- node[right, above] {$\frac{1}{6}$} (yt.north);
        \draw[-{Stealth[length=2.5mm]}] (xt.east) -- (xt+1.west);
        \draw[-{Stealth[length=2.5mm]}] (yt.east) -- (yt+1.west);
        \draw[-{Stealth[length=2.5mm]}] (xt-2.east) -- (xt-1.west);
        \draw[-{Stealth[length=2.5mm]}] (yt-2.east) -- (yt-1.west);
        \draw[-{Stealth[length=2.5mm]}] (xt-2.south) -- (yt-1.north);
        \draw[-{Stealth[length=2.5mm]}] (xt+1.south) --  (bot.north);
        \draw[-{Stealth[length=2.5mm]}] (xt+1.east) -- (top.west);.
        \draw[-{Stealth[length=2.5mm]}] (yt+1.east) -- (bot.west);
        \draw[-{Stealth[length=2.5mm]}] (pasttop.east) -- (xt-2.west);
        \draw[-{Stealth[length=2.5mm]}] (pastbot.east) -- (yt-2.west);
        \draw[-{Stealth[length=2.5mm]}] (pasttop.south) -- (yt-2.north);
    \end{tikzpicture}
    \caption[Full-time DAG  of the process in \eqref{eq:ExampleDsepWith} on its representation without instantaneous effects.]{Full-time DAG $\mathcal{G}_{II}$ of the time series without instantaneous effects defined in Equation \eqref{eq:ExampleDsepWithout}. The coefficients of the edges correspond to the coefficients of the process.}
    \label{fig:ExampleDsepWithout}
\end{figure}
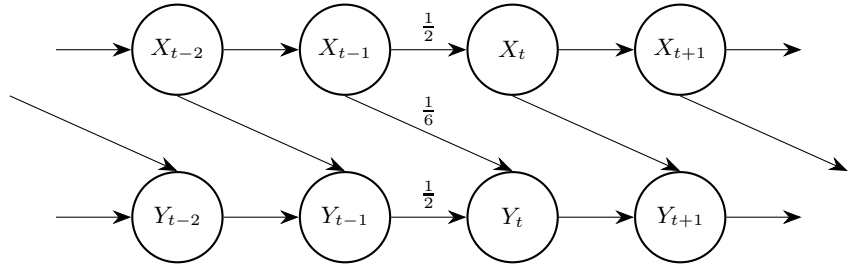 

We see in this example how $d$-separation statements are not inherited by $\mathcal{G}_{II}$ from $\mathcal{G}_I$. Indeed, we have that $Y_t\perp_{\mathcal{G}_I} X_{t-1}|\{X_t,Y_{t-1}\}$, but it does not hold that $Y_t\not\perp_{\mathcal{G}_{II}} X_{t-1}|\{X_t,Y_{t-1}\}$ as there is an edge from $X_{t-1}$ to $Y_t$ in $\mathcal{G}_{II}$. Here we can also see an application of Corollary \ref{coro:Ancestors}, as we have  $AN_{\mathcal{G}_{II}}(X_t)=\{X_s\colon s\leq t\}=AN_{\mathcal{G}_{I}}(X_t)$ and $AN_{\mathcal{G}_{II}}(Y_t)=\{X_{s-1},Y_{s}\colon s\leq t\}\subsetneq\{X_{s},Y_{s}\colon s\leq t\}=AN_{\mathcal{G}_{I}}(Y_t)$.    
\end{example}
As we cannot directly translate $d$-separation statement from $\mathcal{G}_I$ to $\mathcal{G}_{II}$, our approach is to study how the instantaneous effects modify the  correlation structure of the innovations when the process is expressed without instantaneous effects. To that extent, we have the following proposition.
\begin{restatable}{proposition}{propVarianceInnovations}\label{prop:VarianceInnovations}
    Consider a VAR($p$) process with instantaneous effects as per Definition \ref{def:VARWith}, with full-time DAG $\mathcal{G}_I$,  and take the distribution-equivalent VAR($p$) process without instantaneous effects as defined in Equation \eqref{eq:VAREquiv}. If  $S_t^i$ to $S_t^j$  have no common instantaneous ancestors ($AN_{\mathcal{G}_I}(S_t^i)_{[t]}\cap AN_{\mathcal{G}_I}(S_t^j)_{[t]}=\emptyset$), then $\delta_t^i$ and $\delta_t^j$ are independent.
\end{restatable}
Proposition \ref{prop:VarianceInnovations} allows us to identify which innovations in the non-instantaneous effects representation remain independent. This will be key to generalize Theorem \ref{thm:GMPWithout}, for what we present the following Lemma.
\begin{restatable}{lemma}{lemmaGeneralization}\label{lemma:Generalization}
    Consider a VAR($p$) process with instantaneous effects as per Definition \ref{def:VARWith}. Let $\mathcal{G}_I$ be its associated full-time DAG. Let $A,B$ be two disjoint sets of nodes such that they have no common ancestors: $AN_{\mathcal{G}_I}(A)\cap AN_{\mathcal{G}_I}(B)=\emptyset$. Then $A\indep B$.
\end{restatable}
With all these results at hand, we can now finally state what was the first goal of this work: The global Markov property for VAR processes with instantaneous effects.
\begin{restatable}[Global Markov property for VAR processes with instantaneous effects]{theorem}{thmGMPWith}\label{thm:GMPWith} 
    Consider $p\in\mathbb{N}$, and $S$ a  VAR($p$) process with instantaneous effects as defined in \ref{def:VARWith}. Consider finite pairwise disjoint sets $A,B,C$ of nodes of the full-time DAG $\mathcal{G}_{I}$.  If $A$ and $C$ are d-separated by $B$ in $\mathcal{G}_{I}$, then  $A\indep C|B$.
\end{restatable}
The proof of this Theorem, which can be found in Appendix \ref{sec:AppendixProofs}, is an adapted version of the proof of Theorem \ref{thm:GMPWithout} provided in \cite{thams2022identifying}, modified to account for the existence of instantaneous effects.
\subsection{The global Markov property for VARMA processes with instantaneous effects}\label{sec:GMPVarma}
Having dealt with the global Markov property for VAR processes with instantaneous effects, the natural next step is to consider a broader family of time series which includes them: VARMA processes with instantaneous effects.
\begin{definition}[VARMA process with instantaneous effects]\label{def:VARMAWith}
    We say that a  time series $S$ follows a VARMA($p,q$) process with instantaneous effects if there exist $p,q\in\mathbb{N}$ and coefficient matrices $A_0,A_1,\ldots,A_p,B_1,,\ldots,B_q\in\mathbb{R}^{d\times d}$ such that $diag (A_0)=0$ and for all $t\in\mathbb{Z}$: $$S_t=A_0S_t+A_1S_{t-1}+\ldots +A_pS_{t-p}+\varepsilon_t+B_1\varepsilon_{t-1}+\ldots B_q\varepsilon_{t-q},$$ where $A_0,A_1,\ldots,A_p$ are such that $\det ((I_d-A_0)\lambda^p-A_1\lambda^{p-1}-\ldots -A_p)=0$ implies $|\lambda|<1$, and   $\varepsilon_t$ is an i.i.d. process with finite second moments and jointly independent components. We define the full-time graph of this process $\mathcal{G}_{full}$ as an infinite directed graph over nodes $\{S_t^i,\varepsilon_t^i\colon i\in\{1,\ldots,d\},\ t\in\mathbb{Z}\}$ such that:
    \begin{itemize}
        \item For $t\in\mathbb{Z}$ and $k\in\{0,\ldots,p\}$ there is an edge from $S_{t-k}^j$ to $S_t^i$ if $(A_k)_{i,j}\neq 0$,
        \item for $t\in\mathbb{Z}$ and $l\in\{1,\ldots,q\}$ there is an edge from $\varepsilon_{t-l}^j$ to $S_t^i$ if $(B_l)_{i,j}\neq 0$,
        \item for $t\in\mathbb{Z}$ and $i\in\{1,\ldots,d\}$ there is an edge from $\varepsilon_{t}^i$ to $S_t^i$.
    \end{itemize}
\end{definition}
The problem with the full-time DAG of a VARMA process is that it has nodes representing the innovations. This is clearly sub-optimal, as the innovations are an exogenous part of the process. We would  like to construct a graph  over the endogenous nodes which also conveys the information related to the exogenous ones (the innovations). To do so we will use latent projections, introduced in Definition \ref{def:LatentProjection}. The problem with this definition is that it does not apply directly to the full-time DAG of a VARMA process, as this DAG is infinite. Nevertheless, due to the repetitive structure of this DAG, we can consider the latent projection of a large-enough time-frame, and then iterate the structure of the resulting ADMG. We then obtain the full-time marginalized ADMG.
\begin{definition}[Full-time marginalized ADMG of a VARMA process]\label{def:FullTimeMarginalizedADMG}
    Consider a time series $S$ following a VARMA($ p, q$) process as per Definition \ref{def:VARMAWith}, and let $\mathcal{G}_{full}$ be its full-time DAG. Its full-time marginalized ADMG is an infinite ADMG over the nodes $\{S_t^i\colon t\in\mathbb{Z}, i\in\{1,\ldots,d\}\}$ constructed the following way: For an arbitrary $t^*\in\mathbb{Z}$, consider the proper subgraph of $\mathcal{G}_{full}$ over nodes  $\{S_t^i,\varepsilon_t^i\colon t\in\{t^*-\max(p,q)-1,\ldots,t^*\}, i\in\{1,\ldots,d\}\}$, and then its latent projection over $\{S_t^i\colon t\in\{t^*-\max(p,q)-1,\ldots,t^*\}, i\in\{1,\ldots,d\}\}$. The full-time marginalized ADMG is then obtained by infinitely repeating the structure of this ADMG. We denote it by $(\mathcal{G}_{full})_S$, or simply by $\mathcal{G}_S$.
\end{definition}
\begin{remark}
    A more general definition of the marginalization of an infinite full-time DAG can be found in \cite[Def. 4]{thams2022identifying}. While this definition comes in the context of VAR processes, it can be straightforwardly generalized to the full-time DAG of a VARMA process.
\end{remark}
\begin{example}\label{example:GMPVarmaWith}
    Consider the following 2-dimensional VARMA(1,1) process with instantaneous effects: 
    \begin{equation}\label{eq:ExampleVarmaDAGWith}
        \begin{cases}
            X_t=\frac{1}{2}X_{t-1}+\varepsilon_t^X+\frac{1}{4}\varepsilon_{t-1}^Y\\
            Y_t=\frac{1}{5}X_t+ \frac{1}{3}X_{t-1}+\frac{1}{2}Y_{t-1}+\varepsilon_t^Y
        \end{cases},
    \end{equation}
    with $\varepsilon_t\overset{\text{i.i.d.}}{\sim}\mathcal{N}(0,I_2)$. Its associated full-time DAG $\mathcal{G}_{I}$ can be seen in Figure \ref{fig:ExampleVarmaDAGWith}. The full-time marginalized  ADMG associated with the process defined in Equation \eqref{eq:ExampleVarmaDAGWith}, which we denote by $\mathcal{G}_{I,S}$, can be found in Figure \ref{fig:ExampleVarmaADMGWith}. We see how this graph has both directed and bi-directed edges. 

    As explained before, the process introduced in Equation \eqref{eq:ExampleVarmaDAGWith} can be expressed in a distributional-equivalent form without instantaneous effects, as: 
    \begin{equation}
        \begin{cases}
            X_t=\frac{1}{2}X_{t-1}+\varepsilon_t^X+\frac{1}{4}\varepsilon_{t-1}^Y\\
            Y_t= \frac{13}{30}X_{t-1}+\frac{1}{2}Y_{t-1}+\frac{1}{5}\varepsilon_t^X +\varepsilon_t^Y+\frac{1}{20}\varepsilon^Y_{t-1}
        \end{cases},
        \label{eq:ExampleVarmaDAGWithEquiv}
    \end{equation}
    with associated full time DAG $\mathcal{G}_{II}$ which can be seen in Figure \ref{fig:ExampleVarmaDAGWithEquiv}, which marginalizes to the full-time ADMG $\mathcal{G}_{II,S}$ seen in Figure \ref{fig:ExampleVarmaADMGWithEquiv}.
    
    It is key to note how Corollary \ref{coro:Ancestors} still applies to these full-time DAGs, and therefore it also extends to the marginalized full-time ADMGs. Also worth mentioning is the fact that the ADMG in Figure \ref{fig:ExampleVarmaADMGWithEquiv} has some bi-directed edges which do not exist in the ADMG in Figure \ref{fig:ExampleVarmaADMGWith}. This owes to the fact that when we rewrite the process without instantaneous effects we introduce some direct edges coming from innovations which were previously only related to instantaneous ancestors, yielding new bi-directed edges. 
    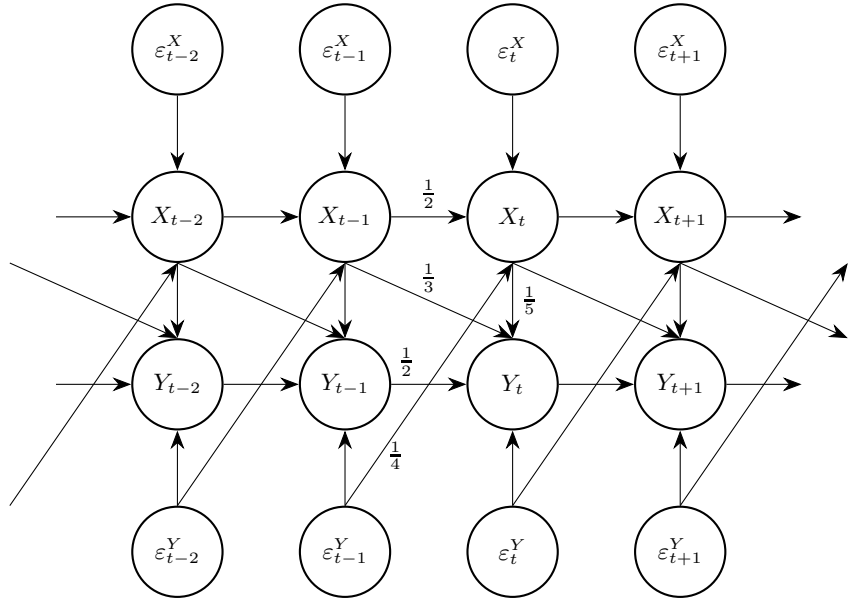
\begin{figure}[ht]
        \centering
        \begin{tikzpicture}[
        main/.style={draw, circle, thick, minimum size=12mm}
        ]
        \node[main] (xt) {$X_t$};
        \node[main] (yt) [below = of xt] {$Y_t$};
        \node[main] (xt-1) [left =  of xt]{$X_{t-1}$};
        \node[main] (yt-1) [below = of xt-1] {$Y_{t-1}$};
        \node[main] (xt-2) [left =  of xt-1]{$X_{t-2}$};
        \node[main] (yt-2) [below = of xt-2] {$Y_{t-2}$};
        \node[main] (xt+1) [right =  of xt]{$X_{t+1}$};
        \node[main] (yt+1) [below = of xt+1] {$Y_{t+1}$};
        \node[main, draw=none] (top) [right = of xt+1] {};
        \node[main, draw=none] (bot) [below = of top] {};
        \node[main, draw=none] (pasttop) [left = of xt-2] {};
        \node[main, draw=none] (pastbot) [below = of pasttop] {};
        \node[main] (ext)[above = of xt] {$\varepsilon^X_t$};
        \node[main] (eyt) [below = of yt] {$\varepsilon^Y_t$};
        \node[main] (ext-1) [left =  of ext]{$\varepsilon^X_{t-1}$};
        \node[main] (eyt-1) [below = of yt-1] {$\varepsilon^Y_{t-1}$};
        \node[main] (ext-2) [left =  of ext-1]{$\varepsilon^X_{t-2}$};
        \node[main] (eyt-2) [below = of yt-2] {$\varepsilon^Y_{t-2}$};
        \node[main] (ext+1) [right =  of ext]{$\varepsilon^X_{t+1}$};
        \node[main] (eyt+1) [below = of yt+1] {$\varepsilon^Y_{t+1}$};
        \node[main, draw=none] (epastbot) [below = of pastbot] {};
        
        \draw[-{Stealth[length=2.5mm]}] (xt.south) --  (yt+1.north);
        \draw[-{Stealth[length=2.5mm]}] (xt-1.east) -- node[above] {$\frac{1}{2}$} (xt.west);
        \draw[-{Stealth[length=2.5mm]}] (yt-1.east) -- node[pos=0.2, above] {$\frac{1}{2}$} (yt.west);
        \draw[-{Stealth[length=2.5mm]}] (xt-1.south) -- node[right, above] {$\frac{1}{3}$} (yt.north);
        \draw[-{Stealth[length=2.5mm]}] (xt.south) -- node[right] {$\frac{1}{5}$} (yt.north);
        \draw[-{Stealth[length=2.5mm]}] (xt.east) -- (xt+1.west);
        \draw[-{Stealth[length=2.5mm]}] (yt.east) -- (yt+1.west);
        \draw[-{Stealth[length=2.5mm]}] (xt-2.east) -- (xt-1.west);
        \draw[-{Stealth[length=2.5mm]}] (yt-2.east) -- (yt-1.west);
        \draw[-{Stealth[length=2.5mm]}] (xt-2.south) -- (yt-1.north);
        \draw[-{Stealth[length=2.5mm]}] (xt+1.south) --  (bot.north);
        \draw[-{Stealth[length=2.5mm]}] (xt+1.east) -- (top.west);.
        \draw[-{Stealth[length=2.5mm]}] (yt+1.east) -- (bot.west);
        \draw[-{Stealth[length=2.5mm]}] (pasttop.east) -- (xt-2.west);
        \draw[-{Stealth[length=2.5mm]}] (pastbot.east) -- (yt-2.west);
        \draw[-{Stealth[length=2.5mm]}] (pasttop.south) -- (yt-2.north);
        \draw[-{Stealth[length=2.5mm]}] (xt-1.south) -- (yt-1.north);
        \draw[-{Stealth[length=2.5mm]}] (xt-2.south) -- (yt-2.north);
        \draw[-{Stealth[length=2.5mm]}] (xt+1.south) -- (yt+1.north);
        \draw[-{Stealth[length=2.5mm]}] (ext.south) --  (xt.north);
        \draw[-{Stealth[length=2.5mm]}] (ext-1.south) --  (xt-1.north);
        \draw[-{Stealth[length=2.5mm]}] (ext-2.south) --  (xt-2.north);
        \draw[-{Stealth[length=2.5mm]}] (ext+1.south) --  (xt+1.north);
        \draw[-{Stealth[length=2.5mm]}] (eyt.north) --  (yt.south);
        \draw[-{Stealth[length=2.5mm]}] (eyt-1.north) --  (yt-1.south);
        \draw[-{Stealth[length=2.5mm]}] (eyt-2.north) --  (yt-2.south);
        \draw[-{Stealth[length=2.5mm]}] (eyt+1.north) --  (yt+1.south);
        \draw[-{Stealth[length=2.5mm]}] (eyt.north) --  (xt+1.south);
        \draw[-{Stealth[length=2.5mm]}] (eyt-1.north) -- node[pos=0.3,below] {$\frac{1}{4}$}  (xt.south);
        \draw[-{Stealth[length=2.5mm]}] (eyt-2.north) --  (xt-1.south);
        \draw[-{Stealth[length=2.5mm]}] (epastbot.north) --  (xt-2.south);
        \draw[-{Stealth[length=2.5mm]}] (eyt+1.north) --  (top.south);
        
    \end{tikzpicture}
    \caption[Full-time DAG  of the process  defined in  \eqref{eq:ExampleVarmaDAGWith}.]{Full-time DAG $\mathcal{G}_I$ of the time series  defined in Equation \eqref{eq:ExampleVarmaDAGWith}. The coefficients of the edges correspond to the coefficients of the process.}
    \label{fig:ExampleVarmaDAGWith}
    \end{figure}
    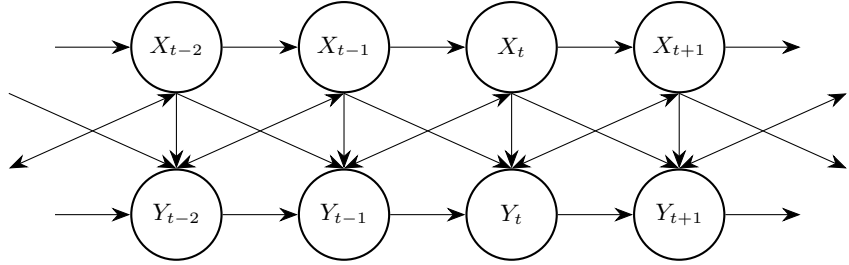
\begin{figure}[ht]
        \centering
        \begin{tikzpicture}[
        main/.style={draw, circle, thick, minimum size=12mm}
        ]
        \node[main] (xt) {$X_t$};
        \node[main] (yt) [below = of xt] {$Y_t$};
        \node[main] (xt-1) [left =  of xt]{$X_{t-1}$};
        \node[main] (yt-1) [below = of xt-1] {$Y_{t-1}$};
        \node[main] (xt-2) [left =  of xt-1]{$X_{t-2}$};
        \node[main] (yt-2) [below = of xt-2] {$Y_{t-2}$};
        
        \node[main] (xt+1) [right =  of xt]{$X_{t+1}$};
        \node[main] (yt+1) [below = of xt+1] {$Y_{t+1}$};
        \node[main, draw=none] (top) [right = of xt+1] {};
        \node[main, draw=none] (bot) [below = of top] {};
        \node[main, draw=none] (pasttop) [left = of xt-2] {};
        \node[main, draw=none] (pastbot) [below = of pasttop] {};
        \draw[-{Stealth[length=2.5mm]}] (xt.south) --  (yt+1.north);
        \draw[-{Stealth[length=2.5mm]}] (xt-1.east) --  (xt.west);
        \draw[-{Stealth[length=2.5mm]}] (yt-1.east) --  (yt.west);
        \draw[-{Stealth[length=2.5mm]}] (xt-1.south) --  (yt.north);
        \draw[-{Stealth[length=2.5mm]}] (xt.east) -- (xt+1.west);
        \draw[-{Stealth[length=2.5mm]}] (yt.east) -- (yt+1.west);
        \draw[-{Stealth[length=2.5mm]}] (xt-2.east) -- (xt-1.west);
        \draw[-{Stealth[length=2.5mm]}] (yt-2.east) -- (yt-1.west);
        \draw[-{Stealth[length=2.5mm]}] (xt-2.south) -- (yt-1.north);
        \draw[-{Stealth[length=2.5mm]}] (xt+1.south) --  (bot.north);
        \draw[-{Stealth[length=2.5mm]}] (xt+1.east) -- (top.west);.
        \draw[-{Stealth[length=2.5mm]}] (yt+1.east) -- (bot.west);
        \draw[-{Stealth[length=2.5mm]}] (pasttop.east) -- (xt-2.west);
        \draw[-{Stealth[length=2.5mm]}] (pastbot.east) -- (yt-2.west);
        \draw[-{Stealth[length=2.5mm]}] (pasttop.south) -- (yt-2.north);
        \draw[{Stealth[length=2.5mm]}-{Stealth[length=2.5mm]}] (yt-1.north) --  (xt.south);
        \draw[{Stealth[length=2.5mm]}-{Stealth[length=2.5mm]}] (yt-2.north) --  (xt-1.south);
        \draw[{Stealth[length=2.5mm]}-{Stealth[length=2.5mm]}] (yt.north) --  (xt+1.south);
        \draw[{Stealth[length=2.5mm]}-{Stealth[length=2.5mm]}] (yt+1.north) --  (top.south);
        \draw[{Stealth[length=2.5mm]}-{Stealth[length=2.5mm]}] (pastbot.north) --  (xt-2.south);
        \draw[-{Stealth[length=2.5mm]}] (xt.south) -- (yt.north);
        \draw[-{Stealth[length=2.5mm]}] (xt-1.south) -- (yt-1.north);
        \draw[-{Stealth[length=2.5mm]}] (xt-2.south) -- (yt-2.north);
        \draw[-{Stealth[length=2.5mm]}] (xt+1.south) -- (yt+1.north);
    \end{tikzpicture}
    \caption[Full-time marginalized  ADMG  of the process in \eqref{eq:ExampleVarmaDAGWith}.]{Full-time marginalized ADMG $\mathcal{G}_{I,S}$ of the time series  defined in Equation \eqref{eq:ExampleVarmaDAGWith}.}
    \label{fig:ExampleVarmaADMGWith}
    \end{figure}
    \begin{figure}[ht]
        \centering
        \begin{tikzpicture}[
        main/.style={draw, circle, thick, minimum size=12mm}
        ]
        \node[main] (xt) {$X_t$};
        \node[main] (yt) [below = of xt] {$Y_t$};
        \node[main] (xt-1) [left =  of xt]{$X_{t-1}$};
        \node[main] (yt-1) [below = of xt-1] {$Y_{t-1}$};
        \node[main] (xt-2) [left =  of xt-1]{$X_{t-2}$};
        \node[main] (yt-2) [below = of xt-2] {$Y_{t-2}$};
        \node[main] (xt+1) [right =  of xt]{$X_{t+1}$};
        \node[main] (yt+1) [below = of xt+1] {$Y_{t+1}$};
        \node[main, draw=none] (top) [right = of xt+1] {};
        \node[main, draw=none] (bot) [below = of top] {};
        \node[main, draw=none] (pasttop) [left = of xt-2] {};
        \node[main, draw=none] (pastbot) [below = of pasttop] {};
        \node[main] (ext)[above = of xt] {$\varepsilon^X_t$};
        \node[main] (eyt) [below = of yt] {$\varepsilon^Y_t$};
        \node[main] (ext-1) [left =  of ext]{$\varepsilon^X_{t-1}$};
        \node[main] (eyt-1) [below = of yt-1] {$\varepsilon^Y_{t-1}$};
        \node[main] (ext-2) [left =  of ext-1]{$\varepsilon^X_{t-2}$};
        \node[main] (eyt-2) [below = of yt-2] {$\varepsilon^Y_{t-2}$};
        \node[main] (ext+1) [right =  of ext]{$\varepsilon^X_{t+1}$};
        \node[main] (eyt+1) [below = of yt+1] {$\varepsilon^Y_{t+1}$};
        \node[main, draw=none] (epastbot) [below = of pastbot] {};
        
        \draw[-{Stealth[length=2.5mm]}] (xt.south) --  (yt+1.north);
        \draw[-{Stealth[length=2.5mm]}] (xt-1.east) -- node[above] {$\frac{1}{2}$} (xt.west);
        \draw[-{Stealth[length=2.5mm]}] (yt-1.east) -- node[pos=0.2, above] {$\frac{1}{2}$} (yt.west);
        \draw[-{Stealth[length=2.5mm]}] (xt-1.south) -- node[right, above] {$\frac{13}{30}$} (yt.north);
        \draw[-{Stealth[length=2.5mm]}] (xt.east) -- (xt+1.west);
        \draw[-{Stealth[length=2.5mm]}] (yt.east) -- (yt+1.west);
        \draw[-{Stealth[length=2.5mm]}] (xt-2.east) -- (xt-1.west);
        \draw[-{Stealth[length=2.5mm]}] (yt-2.east) -- (yt-1.west);
        \draw[-{Stealth[length=2.5mm]}] (xt-2.south) -- (yt-1.north);
        \draw[-{Stealth[length=2.5mm]}] (xt+1.south) --  (bot.north);
        \draw[-{Stealth[length=2.5mm]}] (xt+1.east) -- (top.west);.
        \draw[-{Stealth[length=2.5mm]}] (yt+1.east) -- (bot.west);
        \draw[-{Stealth[length=2.5mm]}] (pasttop.east) -- (xt-2.west);
        \draw[-{Stealth[length=2.5mm]}] (pastbot.east) -- (yt-2.west);
        \draw[-{Stealth[length=2.5mm]}] (pasttop.south) -- (yt-2.north);
        
        \draw[-{Stealth[length=2.5mm]}] (ext.south) --  (xt.north);
        \draw[-{Stealth[length=2.5mm]}] (ext-1.south) --  (xt-1.north);
        \draw[-{Stealth[length=2.5mm]}] (ext-2.south) --  (xt-2.north);
        \draw[-{Stealth[length=2.5mm]}] (ext+1.south) --  (xt+1.north);
        \draw[-{Stealth[length=2.5mm]}] (eyt.north) --  (yt.south);
        \draw[-{Stealth[length=2.5mm]}] (eyt-1.north) --  (yt-1.south);
        \draw[-{Stealth[length=2.5mm]}] (eyt-2.north) --  (yt-2.south);
        \draw[-{Stealth[length=2.5mm]}] (eyt+1.north) --  (yt+1.south);
        \draw[-{Stealth[length=2.5mm]}] (eyt.north) --  (xt+1.south);
        \draw[-{Stealth[length=2.5mm]}] (eyt-1.north) -- node[pos=0.4,below] {$\frac{1}{4}$}  (xt.south);
        \draw[-{Stealth[length=2.5mm]}] (eyt-2.north) --  (xt-1.south);
        \draw[-{Stealth[length=2.5mm]}] (epastbot.north) --  (xt-2.south);
        \draw[-{Stealth[length=2.5mm]}] (eyt+1.north) --  (top.south);
        \draw[-{Stealth[length=2.5mm]}] (ext.south) to [out=335,in=25] node[right,pos=0.2]{$\frac{1}{5}$}  (yt.north);
        \draw[-{Stealth[length=2.5mm]}] (ext-1.south) to [out=335,in=25]   (yt-1.north);
        \draw[-{Stealth[length=2.5mm]}] (ext-2.south) to [out=335,in=25]   (yt-2.north);
        \draw[-{Stealth[length=2.5mm]}] (ext+1.south) to [out=335,in=25]   (yt+1.north);
        \draw[-{Stealth[length=2.5mm]}] (eyt-1.north) -- node[pos=0.4,below] {$\frac{1}{20}$} (yt.south);
        \draw[-{Stealth[length=2.5mm]}] (eyt-2.north) --  (yt-1.south);
        \draw[-{Stealth[length=2.5mm]}] (eyt.north) --  (yt+1.south);
        \draw[-{Stealth[length=2.5mm]}] (eyt+1.north) --  (bot.south);
        \draw[-{Stealth[length=2.5mm]}] (epastbot.north) --  (yt-2.south);
    \end{tikzpicture}
    \caption[Full-time DAG  of the process  defined in  \eqref{eq:ExampleVarmaDAGWithEquiv}.]{Full-time DAG $\mathcal{G}_{II}$ of the time series  defined in Equation \eqref{eq:ExampleVarmaDAGWithEquiv}. The coefficients of the edges correspond to the coefficients of the process.}
    \label{fig:ExampleVarmaDAGWithEquiv}
    \end{figure}
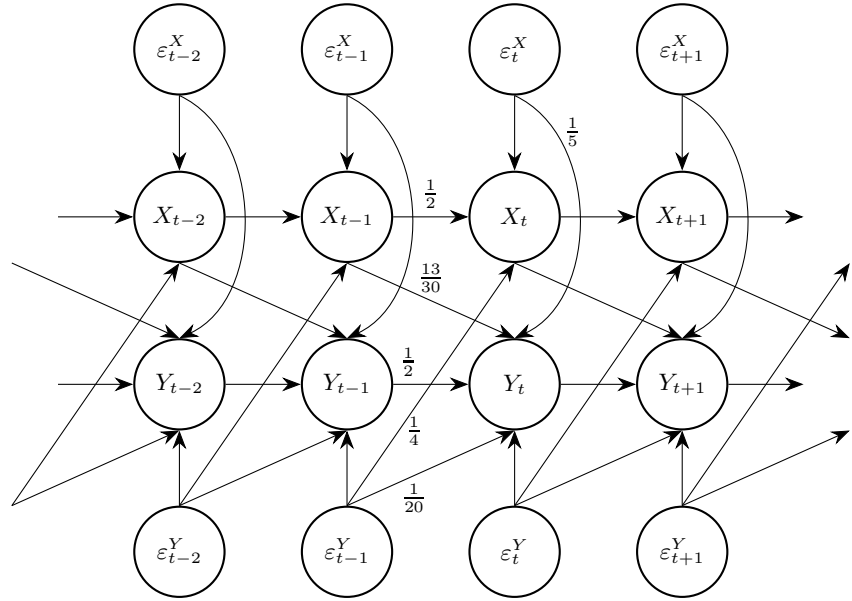
    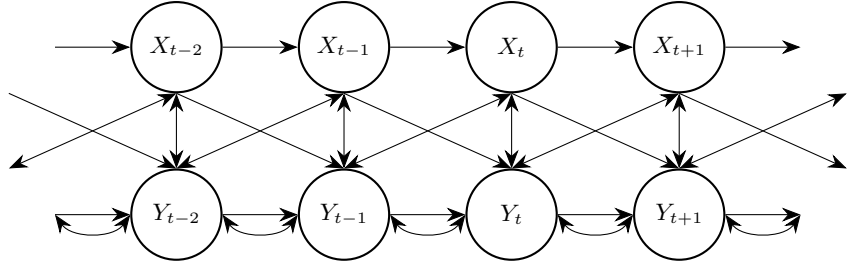
\begin{figure}[ht]
        \centering
        \begin{tikzpicture}[
        main/.style={draw, circle, thick, minimum size=12mm}
        ]
        \node[main] (xt) {$X_t$};
        \node[main] (yt) [below = of xt] {$Y_t$};
        \node[main] (xt-1) [left =  of xt]{$X_{t-1}$};
        \node[main] (yt-1) [below = of xt-1] {$Y_{t-1}$};
        \node[main] (xt-2) [left =  of xt-1]{$X_{t-2}$};
        \node[main] (yt-2) [below = of xt-2] {$Y_{t-2}$};
        
        \node[main] (xt+1) [right =  of xt]{$X_{t+1}$};
        \node[main] (yt+1) [below = of xt+1] {$Y_{t+1}$};
        \node[main, draw=none] (top) [right = of xt+1] {};
        \node[main, draw=none] (bot) [below = of top] {};
        \node[main, draw=none] (pasttop) [left = of xt-2] {};
        \node[main, draw=none] (pastbot) [below = of pasttop] {};
        \draw[-{Stealth[length=2.5mm]}] (xt.south) --  (yt+1.north);
        \draw[-{Stealth[length=2.5mm]}] (xt-1.east) --  (xt.west);
        \draw[-{Stealth[length=2.5mm]}] (yt-1.east) --  (yt.west);
        \draw[-{Stealth[length=2.5mm]}] (xt-1.south) --  (yt.north);
        \draw[-{Stealth[length=2.5mm]}] (xt.east) -- (xt+1.west);
        \draw[-{Stealth[length=2.5mm]}] (yt.east) -- (yt+1.west);
        \draw[-{Stealth[length=2.5mm]}] (xt-2.east) -- (xt-1.west);
        \draw[-{Stealth[length=2.5mm]}] (yt-2.east) -- (yt-1.west);
        \draw[-{Stealth[length=2.5mm]}] (xt-2.south) -- (yt-1.north);
        \draw[-{Stealth[length=2.5mm]}] (xt+1.south) --  (bot.north);
        \draw[-{Stealth[length=2.5mm]}] (xt+1.east) -- (top.west);.
        \draw[-{Stealth[length=2.5mm]}] (yt+1.east) -- (bot.west);
        \draw[-{Stealth[length=2.5mm]}] (pasttop.east) -- (xt-2.west);
        \draw[-{Stealth[length=2.5mm]}] (pastbot.east) -- (yt-2.west);
        \draw[-{Stealth[length=2.5mm]}] (pasttop.south) -- (yt-2.north);
        \draw[{Stealth[length=2.5mm]}-{Stealth[length=2.5mm]}] (yt-1.north) --  (xt.south);
        \draw[{Stealth[length=2.5mm]}-{Stealth[length=2.5mm]}] (yt-2.north) --  (xt-1.south);
        \draw[{Stealth[length=2.5mm]}-{Stealth[length=2.5mm]}] (yt.north) --  (xt+1.south);
        \draw[{Stealth[length=2.5mm]}-{Stealth[length=2.5mm]}] (yt+1.north) --  (top.south);
        \draw[{Stealth[length=2.5mm]}-{Stealth[length=2.5mm]}] (pastbot.north) --  (xt-2.south);
        \draw[{Stealth[length=2.5mm]}-{Stealth[length=2.5mm]}] (xt.south) -- (yt.north);
        \draw[{Stealth[length=2.5mm]}-{Stealth[length=2.5mm]}] (xt-1.south) -- (yt-1.north);
        \draw[{Stealth[length=2.5mm]}-{Stealth[length=2.5mm]}] (xt-2.south) -- (yt-2.north);
        \draw[{Stealth[length=2.5mm]}-{Stealth[length=2.5mm]}] (xt+1.south) -- (yt+1.north);
        \draw[{Stealth[length=2.5mm]}-{Stealth[length=2.5mm]}] (yt-1.east) to [out=305,in=235]  (yt.west);
        \draw[{Stealth[length=2.5mm]}-{Stealth[length=2.5mm]}] (yt-2.east) to [out=305,in=235]  (yt-1.west);
        \draw[{Stealth[length=2.5mm]}-{Stealth[length=2.5mm]}] (yt.east) to [out=305,in=235]  (yt+1.west);
        \draw[{Stealth[length=2.5mm]}-{Stealth[length=2.5mm]}] (yt+1.east) to [out=305,in=235]  (bot.west);
        \draw[{Stealth[length=2.5mm]}-{Stealth[length=2.5mm]}] (pastbot.east) to [out=305,in=235]  (yt-2.west);
    \end{tikzpicture}
    \caption[Full-time marginalized  ADMG  of the process in \eqref{eq:ExampleVarmaDAGWithEquiv}.]{Full-time marginalized  ADMG $\mathcal{G}_{II,S}$ of the time series  defined in Equation \eqref{eq:ExampleVarmaDAGWithEquiv}.}
    \label{fig:ExampleVarmaADMGWithEquiv}
    \end{figure}
\end{example}
Using the notion of $m$ separation introduced in Section \ref{sec:PrelimADMG}, we can check separation statements in the full-time marginalized ADMG of a VARMA process. These $m$-separations will translate into conditional independencies in the stationary distribution of the process, as stated in the next theorem.
\begin{restatable}[Global Markov property for VARMA processes with instantaneous effects]{theorem}{thmGmpVarmaWith} \label{thm:GmpVarmaWith}
    Consider $p,q\in\mathbb{N}$, and $S$ a  VARMA($p,q$) process with instantaneous effects as per Definition \ref{def:VARMAWith}.  Consider finite pairwise disjoint sets $A,B,C$ of nodes of the marginalized full-time ADMG $\mathcal{G}_S$.  If $A$ and $C$ are $m$-separated by $B$ in $\mathcal{G}_S$, then  $A\indep C|B$.
\end{restatable}

\subsection{A comment on \cite{hochsprungglobal}}\label{sec:Hochsprung}
Contemporaneous to the writing of this thesis, the  work by \cite{hochsprungglobal} was published. In this article, the authors prove the global Markov property for a certain family of time series with respect to a full-time DAG, under a series of assumptions. Their work encompasses a somewhat different family of time series than the ones considered in this work, as the structural equations of the processes they consider need not be linear nor have additive noise, see \cite[Eq. (1)]{hochsprungglobal}. Additionally, they allow for instantaneous effects. Of particular interest to us is the fact that one of their applications is the global Markov property for VAR processes with instantaneous effects. Their general result boils down to Theorem \ref{thm:GMPWith}. However, there is a subtlety. One of the conditions required in \cite{hochsprungglobal} is that the time series considered is $\alpha$-mixing. For strictly stationary VAR processes, a sufficient condition for $\alpha$-mixing is that the innovations have an absolutely continuous distribution w.r.t. the Lebesgue measure, as shown in \cite{mokkadem1988mixing}. Furthermore, there are stationary VAR processes such that the absolute continuity assumption does not hold and which are not $\alpha$-mixing (an example can be found in \cite[Sec. 2.3]{doukhan2012mixing}). Therefore, Theorem \ref{thm:GMPWith} is more general than the result presented in \cite[Sec. 4]{hochsprungglobal}, as it does not require the absolute continuity of the innovation distribution. 

Furthermore, \cite[Eq. (1)]{hochsprungglobal} restricts  the process to the setting where the only innovation which can enter the structural equation of $S_t^i$ is $\varepsilon_t^i$, with $\varepsilon_t$ being an i.i.d. process with jointly independent components. Therefore, the global Markov property for VARMA models with instantaneous effects (Theorem \ref{thm:GmpVarmaWith}) is not a consequence of  \cite[Thm. 1]{hochsprungglobal}, even if the absolute-continuity assumption of the innovation distribution were added. Thus, the families of time series for which \cite{hochsprungglobal} and this thesis study the global Markov property are notably different.
\section{Faithfulness for VARMA processes with instantaneous effects}\label{chap:faithfulness}
The global Markov property studied in Section \ref{sec:GMP} allows us to translate separation statements in a graph into conditional independencies. As usual in mathematics, we are interested as well on investigating whether the reciprocate statement holds, meaning whether conditional independencies imply separation. This would mean that the conditional independencies in the joint distribution of the process are uniquely the ones implied by the global Markov condition. This property is referred to as {faithfulness}:
\begin{definition}[Faithfulness of a distribution to a graph (adapted from \cite{pearl2000causality})]\label{def:FaithfulnessDAG}
    Consider $\mathcal{G}$ a graph over nodes $V$. We say that a distribution $\mathbb{P}$ over $V$ is faithful to  $\mathcal{G}$ if for every $A,B,C$ finite pairwise disjoint sets of nodes,  $A\indep C|B$ in $\mathbb{P}$ implies $A\perp_{\mathcal{G}}C|B$ (for a particular separation criterion).
\end{definition}
Violation of faithfulness in  linear systems is associated with coefficient cancellation, as we can see in Example \ref{ex:FaithfulnessViolation} for the case where the graph is a DAG and we consider $d$-separation statements.
\begin{example}[Faithfulness violation in a  linear system]\label{ex:FaithfulnessViolation}
    Consider the following structural equation model: 
    \begin{equation}\label{eq:FaithfulnessViolationExampleSCM}
        \begin{cases}
        X = \varepsilon_X\\ Y=\alpha X+\varepsilon_Y\\Z=\gamma X+\beta Y +\varepsilon_Z
    \end{cases},\text{  with   }\varepsilon_X,\varepsilon_Y,\varepsilon_Z\text{ i.i.d.   and   }\alpha,\beta,\gamma\in\mathbb{R}.
    \end{equation}
    Its associated DAG can be found in Figure \ref{fig:FaithfulnessViolationExampleSCM}. Note that as long as $\gamma\neq 0$ or $\alpha\beta\neq 0$ then $X\not\perp Z$. However,  one can write $Z=(\gamma+\alpha\beta)X+\beta\varepsilon_Y+\varepsilon_Z$, and therefore  if $\gamma+\alpha\beta=0$ then $Z\indep X$ even though $Z\not\perp X$, what would constitute a faithfulness violation. 
    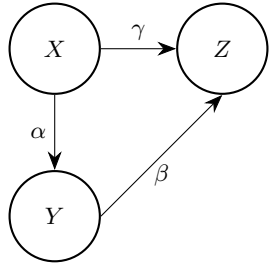
\begin{figure}[ht!]
        \centering
        \begin{tikzpicture}[
        main/.style={draw, circle, thick, minimum size=12mm}
        ]
        \node[main] (x) {$X$};
        \node[main] (y) [below = of x] {$Y$};
        \node[main] (z) [right =  of x]{$Z$};
        
        \draw[-{Stealth[length=2.5mm]}] (x.east) -- node[pos=0.5, above] {$\gamma$} (z.west);
        \draw[-{Stealth[length=2.5mm]}] (x.south) -- node[pos=0.5, left] {$\alpha$} (y.north);
        \draw[-{Stealth[length=2.5mm]}] (y.east) -- node[pos=0.5, below] {$\beta$} (z.south);

    \end{tikzpicture}
    \caption[Associated DAG  of the process  defined in  \eqref{eq:FaithfulnessViolationExampleSCM}.]{Associated  DAG of the process  defined in Equation \eqref{eq:FaithfulnessViolationExampleSCM}. The coefficients of the edges correspond to the coefficients of the process.}
    \label{fig:FaithfulnessViolationExampleSCM}
    \end{figure}
\end{example}

Notice how the violation of faithfulness mentioned in Example \ref{ex:FaithfulnessViolation} comes from imposing a restriction on the coefficients of the process. If, in said example, $(\alpha,\beta,\gamma)$ were sampled on $\mathbb{R}^3$ from a distribution  $\mathbb{P}$ which were absolutely continuous w.r.t. the Lebesgue measure, then the violation of faithfulness studied before would happen with null probability. This was the main idea  motivating the results on faithfulness which can be found in \cite[Thm. 3.2]{spirtes2001causation}, which also motivates the results of this  sections.

\subsection{Faithfulness for VAR processes with instantaneous effects}
Before considering VARMA processes in full generality, we will restrict ourselves to VAR processes with instantaneous effects. Treating this case first, which is not trivial, will considerably simplify the study of the VARMA case. We then consider  first VAR processes with instantaneous effects. 
\begin{restatable}[Faithfulness for VAR processes with instantaneous effects]{theorem}{thmFaithfulnessVARWith}\label{thm:FaithfulnessVARWith}
    Consider $S$ a  VAR($p$) process with instantaneous effects as per Definition \ref{def:VARWith}, such that  $\varepsilon_t\overset{\text{i.i.d.}}{\sim}\mathcal{N}(0,\Gamma)$ with $\Gamma$ a diagonal matrix. Denote the full-time DAG of this process as $\mathcal{G}_{full}$. Assume that the components of the process are ordered according to the topological order of the instantaneous effects, so that $A_0$ is triangular inferior with a zero diagonal. Define: 
    \begin{equation*}
        \begin{split}
            \mathcal{D}\defeq\{&(\text{vech}(A_0),\text{vec}(A_1|\ldots|A_p))\footnotemark
        \in\mathbb{R}^{pd^2+d(d-1)/2}\colon\\& \det((I_d-A_0)\lambda^p-A_1\lambda^{p-1}-\ldots-A_p)=0\Rightarrow |\lambda|<1\}.
        \end{split}
    \end{equation*}
    Assume that the coefficients of the process $(\text{vech}(A_0),\text{vec}(A_1|\ldots|A_p),\sigma_1^2,\ldots,\sigma_d^2)$ (with $\sigma_i^2=\mathbb{V}\text{ar}(\varepsilon^i)$) are sampled from a distribution $\mathbb{P}$ on $\overset{\circ}{\mathcal{D}}\times (0,\infty)^d$ which is absolutely continuous w.r.t. the Lebesgue measure. Then the distribution of the process is almost-surely faithful  to the full-time DAG $\mathcal{G}_{full}$. More precisely, the set of coefficients and variances which lead to a distribution of the process such that there exist $A,B,C$ finite pairwise disjoint sets of nodes such that $A\indep C|B$ but $A\not\perp_{\mathcal{G}_{full}}C|B$ has probability zero w.r.t. $\mathbb{P}$.
    \footnotetext{If $A$ is an $n\times m$ real-valued matrix, we define vec($A$) as the $\mathbb{R}^{nm}$ vector obtained from stacking the columns of $A$ one after another. If $B$ is a $n\times n$ real-valued matrix, we define vech($B$) as the $\mathbb{R}^{n(n-1)/2}$ vector obtained from stacking the elements of $B$ under its diagonal column-wise. This notation is taken from \cite[Sec. 10.2.2]{francq2019garch}. }
\end{restatable}

\subsection{Faithfulness for VARMA processes with instantaneous effects}
To tackle the case of VARMA processes, we follow the same basic idea of the  proof of Theorem \ref{thm:GmpVarmaWith}  presented in  Appendix  \ref{sec:AppendixProofs}: Rewriting the VARMA process as a VAR process of double dimension, and then applying some of the ideas used in the proof of Theorem \ref{thm:FaithfulnessVARWith} (see Appendix \ref{sec:AppendixProofs}). We can the introduce and prove the following result, which culminates this section.
\begin{restatable}[Faithfulness for VARMA processes with instantaneous effects]{theorem}{thmFaithfulnessVARMAWith} \label{thm:FaithfulnessVARMAwith}
    Consider $S$ a  VARMA($p,q$) process with instantaneous effects as per Definition \ref{def:VARMAWith}, such that $A_0$ is triangular inferior with a diagonal of zeroes and $\varepsilon_t\overset{\text{i.i.d.}}{\sim}\mathcal{N}(0,\Gamma)$ with $\Gamma$ a diagonal matrix. Denote the full-time marginalized ADMG of this process as $\mathcal{G}_S$. Define: 
    \begin{equation*}
        \begin{split}\mathcal{D}\defeq\{&
            (\text{vech}(A_0),\text{vec}(A_1|\ldots|A_p|B_1|\ldots|B_q))
        \in\mathbb{R}^{(p+q)d^2+d(d-1)/2}\colon\\ &\det((I_d-A_0)\lambda^p-A_1\lambda^{p-1}-\ldots-A_p)=0\Rightarrow |\lambda|<1\}.
        \end{split}
    \end{equation*}
    Assume that the coefficients of the process $(\text{vech}(A_0),\text{vec}(A_1|\ldots|A_p|B_1|\ldots|B_q),$ $\sigma_1^2,\ldots,\sigma_d^2)$ (with $\sigma_i^2=\mathbb{V}\text{ar}(\varepsilon^i)$) are sampled from a distribution $\mathbb{P}$ on $\overset{\circ}{\mathcal{D}}\times (0,\infty)^d$ which is absolutely continuous w.r.t. the Lebesgue measure. Then the distribution of the process is almost-surely faithful  to the full-time marginalized ADMG. More precisely, the set of coefficients and variances which lead to a distribution of the process such that there exist $A,B,C$ finite pairwise disjoint sets of nodes such that $A\indep C|B$ but $A\not\perp_{\mathcal{G}_{S}}C|B$ has probability zero w.r.t. $\mathbb{P}$.
\end{restatable}
\section{Instrumental variable regression for VARMA processes with instantaneous effects}\label{chap:IV}
The aim of this Section is to  define the concept of total causal effect in the context of a VARMA process with instantaneous effects, find conditions under which it can be identified, and propose a consistent estimator of it. To do so, we need to start by defining the target quantity: Total causal effects.
\begin{definition}[Causal paths, path coefficients and total causal effect for a VARMA process with instantaneous effects (adapted from \texorpdfstring{\cite[Def. 15-16]{thams2022identifying}}))]\label{def:PathCoef} \label{def:TotalCausalEffect}
    Let $S$  be a VARMA($p,q$) process with instantaneous effects as per Definition \ref{def:VARMAWith}. Let $\mathcal{G}_{full}$ be its full-time DAG. Let $Y\defeq S_t^{i_0}$ be a node in $\mathcal{G}_{full}$, and let $\mathcal{X}=({S_{t-l_1}^{i_1}},\ldots,{S_{t-l_m}^{i_m}})^{\top}$  be a finite collection of endogenous nodes in this DAG not containing $Y$. 
    \begin{enumerate}
        \item A directed path from $\mathcal{X}$ to $Y$ in $\mathcal{G}_{full}$ is called $\mathcal{X}$-causal (or simply causal) if it only contains nodes in $\{S_t^i\colon t\in\mathbb{Z}, i\in\{1,\ldots,d\}\}$ (endogenous nodes), and does not intersect $\mathcal{X}$ after the first node. For a $\mathcal{X}$-causal path $\pi$: $\mathcal{X}{\to}\cdots{\to} Y$ we define its path coefficient to be the product of the linear coefficients corresponding to the edges found in said path.
        \item  We define the total causal effect of $\mathcal{X}$ on $Y$ as a vector $\beta\in\mathbb{R}^m$, such that for all $j\in\{1,\ldots,m\}$, $\beta_j$ is the sum of the path coefficients over all $\mathcal{X}$-causal paths from $S_{t-l_j}^{i_j}$ to $Y$.
    \end{enumerate}
\end{definition}
The main goal of this section will be to show identifiability of the total causal effect. However, before doing so, we must introduce some notation regarding conditional covariances for random vectors. 
\begin{definition}[Population and empirical conditional covariance between random vectors \cite{thams2022identifying}]
    Let $A,\ B,\ C$ be three random vectors. We define the covariance matrix between $A$ and $C$ conditional on $B$ as
    $$\mathbb{C}\text{ov}(A,C|B)\defeq\mathbb{C}\text{ov}(A-\mathbb{E}[A|B],C-\mathbb{E}[C|B])=\mathbb{E}[AC^{\top}|B]-\mathbb{E}[A|B]\mathbb{E}[C^{\top}|B],$$
     assuming all these quantities exist. Furthermore, the analogous empirical covariance $\widehat{\mathbb{C}\text{ov}}(A,C|B)$ is defined as the empirical covariance between $A$ and $C$ after regressing out $B$, that is
    $$\widehat{\mathbb{C}\text{ov}}(A,C|B)\defeq\widehat{\mathbb{C}\text{ov}}(A-\widehat{\mathbb{E}}[A|B],C-\widehat{\mathbb{E}}[C|B]).$$
\end{definition}
As mentioned in \cite{thams2022identifying}, this definition of conditional covariance allows to cope even with non-linear relationships. Taking it into account, the following theorem (which extends  Theorem 5  from \cite{thams2022identifying}) allows to identify the total causal effect in the context of a VARMA process with instantaneous effects, and to consistently estimate it. Its proof, presented in Appendix \ref{sec:AppendixProofs} follows the ideas of the proof of \cite[Thm. 5]{thams2022identifying}, but now accounting for the more complex structure of VARMA processes with instantaneous effects. 
\begin{restatable}[{IV} regression for VARMA processes with instantaneous effects]{theorem}{thmIVregressionVARMA}\label{thm:IVregressionVARMA}
    Consider a time series $S$ following a VARMA($p,q$) process with instantaneous effects as per Definition \ref{def:VARMAWith}, with full-time marginalized ADMG $\mathcal{G}_S$. Let $Y$ be a node in $\mathcal{G}_S$ and $\mathcal{X,I,B}$ be pairwise disjoint, finite sets of nodes of $\mathcal{G}_S$ not containing $Y$. Assume the following conditions hold:
    \begin{itemize}
        \item[(1)] $\mathcal{I}$ and $Y$ are $m$-separated by $\mathcal{B}$ in $\mathcal{G}_S$ after removing all outgoing edges from $\mathcal{X}$ which are on a $\mathcal{X}$-causal path from $\mathcal{X}$ to $Y$ in $\mathcal{G}_S$ (see Definition \ref{def:PathCoef}),
        \item[(2)] $AN_{\mathcal{G}_S}(\mathcal{B})\cap SP_{\mathcal{G}_S}(DE_{\mathcal{G}_S}(\mathcal{X}\cup \{Y\}))=\emptyset$.
     \end{itemize}
     Then it holds that:
     \begin{itemize}
         \item[(i)] The total causal effect $\beta$ of $\mathcal{X}$ on $Y$  satisfies the moment equation
         \begin{equation}\label{eq:IVmomentEquation}
             \mathbb{E}[\mathbb{C}\text{ov}(Y-\beta\mathcal{X},\mathcal{I}|\mathcal{B})]=0.
         \end{equation}
     \end{itemize}
    Furthermore, if the next condition holds as well
     \begin{itemize}
         \item[(3)] $\mathbb{E}[\mathbb{C}\text{ov}(\mathcal{X},\mathcal{I}|\mathcal{B})]$ has rank $\dim(\mathcal{X})$, that is full row rank,
     \end{itemize}
     then  we can identify the total causal effect, that is: 
     \begin{itemize}
         \item[(ii)] $\beta$ is the only solution to Equation \eqref{eq:IVmomentEquation}.
     \end{itemize}
     Additionally,  to consistently estimate $\beta$ we have the following result:
     \begin{itemize}
         \item[(iii)] If \textbf{X, I, B, Y} are observations of $\mathcal{X, I, B}, Y$ and $W$ is a positive definite matrix,  then $$\hat{\beta}\defeq\arg\min_{b\in\mathbb{R}^{\dim(\mathcal{X})}}\lVert \widehat{\mathbb{C}\text{ov}}(Y-b\mathcal{X},\mathcal{I}|\mathcal{B}) \lVert^2_W$$ is a consistent estimator of $\beta$. 
     \end{itemize}
\end{restatable}
\begin{remark}
    In the  case of  $\mathcal{B}=\emptyset$,  conditions \textit{(1)} and \textit{(3)} in Theorem \ref{thm:IVregressionVARMA} reduce to:
    \begin{itemize}
        \item[\textit{(1')}] $\mathcal{I}$ and $Y$ are $m$-separated  in $\mathcal{G}_S$ after removing all  edges outgoing $\mathcal{X}$ which are on a causal path from $\mathcal{X}$ to $Y$ in $\mathcal{G}_S$,
        \item[\textit{(3')}] $\mathbb{C}\text{ov}(\mathcal{X},\mathcal{I})$ has rank $\dim(\mathcal{X})$, that is full row rank.
    \end{itemize}
    Most notably, condition \textit{(2)} is trivially satisfied.
\end{remark}
\begin{remark} \label{remark:IVDimensionality}
    For condition \textit{(3)} to be satisfied, as $\mathbb{E}[\mathbb{C}\text{ov}(\mathcal{X},\mathcal{I}|\mathcal{B})]$ has dimensions $\dim(\mathcal{X})\times \dim(\mathcal{I})$, a necessary condition  is that $\dim(\mathcal{X})\leq \dim(\mathcal{I})$. This motivates the following definition.
\end{remark}
\begin{definition}
     In the context of Theorem \ref{thm:IVregressionVARMA}, if conditions \textit{(1)} and \textit{(2)} hold, but condition \textit{(3)} is violated due to $\dim(\mathcal{X})>\dim(\mathcal{I})$, we say the total causal effect is {under-identified} by the {IV} regression.
\end{definition}
\begin{example}\label{example:IVVarma}
    Consider  the 2-dimensional VARMA(1,1) process defined as:
    \begin{equation}\label{eq:ExampleVarmaDAG}
        \begin{cases}
            X_t=\frac{1}{2}X_{t-1}+\varepsilon_t^X+\frac{1}{4}\varepsilon_{t-1}^Y\\
            Y_t=\frac{1}{3}X_{t-1}+\frac{1}{2}Y_{t-1}+\varepsilon_t^Y
        \end{cases},
    \end{equation}
    with $\varepsilon_t\overset{\text{i.i.d.}}{\sim}\mathcal{N}(0,I_2)$. Its associated full-time DAG $\mathcal{G}_{full}$ can be seen in Figure \ref{fig:ExampleVarmaDAG}, and its associated full-time marginalized ADMG seen in Figure \ref{fig:ExampleVarmaADMG}.  
    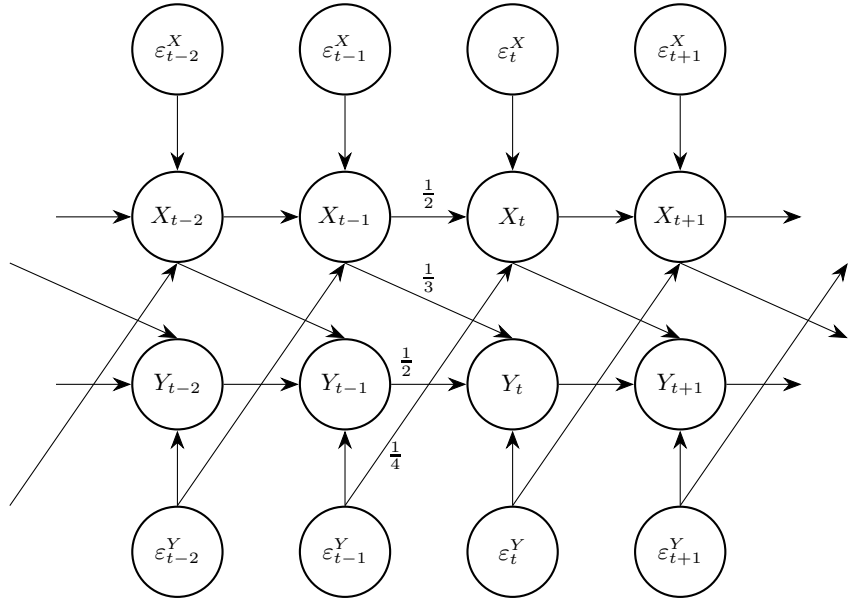
\begin{figure}[ht]
        \centering
        \begin{tikzpicture}[
        main/.style={draw, circle, thick, minimum size=12mm}
        ]
        \node[main] (xt) {$X_t$};
        \node[main] (yt) [below = of xt] {$Y_t$};
        \node[main] (xt-1) [left =  of xt]{$X_{t-1}$};
        \node[main] (yt-1) [below = of xt-1] {$Y_{t-1}$};
        \node[main] (xt-2) [left =  of xt-1]{$X_{t-2}$};
        \node[main] (yt-2) [below = of xt-2] {$Y_{t-2}$};
        \node[main] (xt+1) [right =  of xt]{$X_{t+1}$};
        \node[main] (yt+1) [below = of xt+1] {$Y_{t+1}$};
        \node[main, draw=none] (top) [right = of xt+1] {};
        \node[main, draw=none] (bot) [below = of top] {};
        \node[main, draw=none] (pasttop) [left = of xt-2] {};
        \node[main, draw=none] (pastbot) [below = of pasttop] {};
        \node[main] (ext)[above = of xt] {$\varepsilon^X_t$};
        \node[main] (eyt) [below = of yt] {$\varepsilon^Y_t$};
        \node[main] (ext-1) [left =  of ext]{$\varepsilon^X_{t-1}$};
        \node[main] (eyt-1) [below = of yt-1] {$\varepsilon^Y_{t-1}$};
        \node[main] (ext-2) [left =  of ext-1]{$\varepsilon^X_{t-2}$};
        \node[main] (eyt-2) [below = of yt-2] {$\varepsilon^Y_{t-2}$};
        \node[main] (ext+1) [right =  of ext]{$\varepsilon^X_{t+1}$};
        \node[main] (eyt+1) [below = of yt+1] {$\varepsilon^Y_{t+1}$};
        \node[main, draw=none] (epastbot) [below = of pastbot] {};
        
        \draw[-{Stealth[length=2.5mm]}] (xt.south) --  (yt+1.north);
        \draw[-{Stealth[length=2.5mm]}] (xt-1.east) -- node[above] {$\frac{1}{2}$} (xt.west);
        \draw[-{Stealth[length=2.5mm]}] (yt-1.east) -- node[pos=0.2, above] {$\frac{1}{2}$} (yt.west);
        \draw[-{Stealth[length=2.5mm]}] (xt-1.south) -- node[right, above] {$\frac{1}{3}$} (yt.north);
        \draw[-{Stealth[length=2.5mm]}] (xt.east) -- (xt+1.west);
        \draw[-{Stealth[length=2.5mm]}] (yt.east) -- (yt+1.west);
        \draw[-{Stealth[length=2.5mm]}] (xt-2.east) -- (xt-1.west);
        \draw[-{Stealth[length=2.5mm]}] (yt-2.east) -- (yt-1.west);
        \draw[-{Stealth[length=2.5mm]}] (xt-2.south) -- (yt-1.north);
        \draw[-{Stealth[length=2.5mm]}] (xt+1.south) --  (bot.north);
        \draw[-{Stealth[length=2.5mm]}] (xt+1.east) -- (top.west);.
        \draw[-{Stealth[length=2.5mm]}] (yt+1.east) -- (bot.west);
        \draw[-{Stealth[length=2.5mm]}] (pasttop.east) -- (xt-2.west);
        \draw[-{Stealth[length=2.5mm]}] (pastbot.east) -- (yt-2.west);
        \draw[-{Stealth[length=2.5mm]}] (pasttop.south) -- (yt-2.north);
        \draw[-{Stealth[length=2.5mm]}] (ext.south) --  (xt.north);
        \draw[-{Stealth[length=2.5mm]}] (ext-1.south) --  (xt-1.north);
        \draw[-{Stealth[length=2.5mm]}] (ext-2.south) --  (xt-2.north);
        \draw[-{Stealth[length=2.5mm]}] (ext+1.south) --  (xt+1.north);
        \draw[-{Stealth[length=2.5mm]}] (eyt.north) --  (yt.south);
        \draw[-{Stealth[length=2.5mm]}] (eyt-1.north) --  (yt-1.south);
        \draw[-{Stealth[length=2.5mm]}] (eyt-2.north) --  (yt-2.south);
        \draw[-{Stealth[length=2.5mm]}] (eyt+1.north) --  (yt+1.south);
        \draw[-{Stealth[length=2.5mm]}] (eyt.north) --  (xt+1.south);
        \draw[-{Stealth[length=2.5mm]}] (eyt-1.north) -- node[pos=0.3,below] {$\frac{1}{4}$}  (xt.south);
        \draw[-{Stealth[length=2.5mm]}] (eyt-2.north) --  (xt-1.south);
        \draw[-{Stealth[length=2.5mm]}] (epastbot.north) --  (xt-2.south);
        \draw[-{Stealth[length=2.5mm]}] (eyt+1.north) --  (top.south);
        
    \end{tikzpicture}
    \caption[Full-time DAG  of the process  defined in  \eqref{eq:ExampleVarmaDAG}.]{Full-time DAG $\mathcal{G}_{full}$ of the time series  defined in Equation \eqref{eq:ExampleVarmaDAG}. The coefficients of the edges correspond to the coefficients of the process.}
    \label{fig:ExampleVarmaDAG}
    \end{figure}
    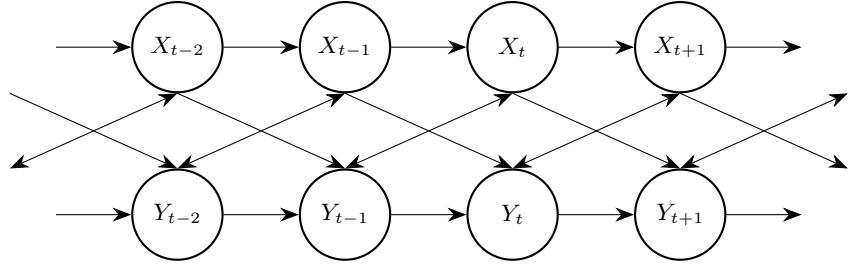
\begin{figure}[ht]
        \centering
        \begin{tikzpicture}[
        main/.style={draw, circle, thick, minimum size=12mm}
        ]
        \node[main] (xt) {$X_t$};
        \node[main] (yt) [below = of xt] {$Y_t$};
        \node[main] (xt-1) [left =  of xt]{$X_{t-1}$};
        \node[main] (yt-1) [below = of xt-1] {$Y_{t-1}$};
        \node[main] (xt-2) [left =  of xt-1]{$X_{t-2}$};
        \node[main] (yt-2) [below = of xt-2] {$Y_{t-2}$};
        
        \node[main] (xt+1) [right =  of xt]{$X_{t+1}$};
        \node[main] (yt+1) [below = of xt+1] {$Y_{t+1}$};
        \node[main, draw=none] (top) [right = of xt+1] {};
        \node[main, draw=none] (bot) [below = of top] {};
        \node[main, draw=none] (pasttop) [left = of xt-2] {};
        \node[main, draw=none] (pastbot) [below = of pasttop] {};
        \draw[-{Stealth[length=2.5mm]}] (xt.south) --  (yt+1.north);
        \draw[-{Stealth[length=2.5mm]}] (xt-1.east) --  (xt.west);
        \draw[-{Stealth[length=2.5mm]}] (yt-1.east) --  (yt.west);
        \draw[-{Stealth[length=2.5mm]}] (xt-1.south) --  (yt.north);
        \draw[-{Stealth[length=2.5mm]}] (xt.east) -- (xt+1.west);
        \draw[-{Stealth[length=2.5mm]}] (yt.east) -- (yt+1.west);
        \draw[-{Stealth[length=2.5mm]}] (xt-2.east) -- (xt-1.west);
        \draw[-{Stealth[length=2.5mm]}] (yt-2.east) -- (yt-1.west);
        \draw[-{Stealth[length=2.5mm]}] (xt-2.south) -- (yt-1.north);
        \draw[-{Stealth[length=2.5mm]}] (xt+1.south) --  (bot.north);
        \draw[-{Stealth[length=2.5mm]}] (xt+1.east) -- (top.west);.
        \draw[-{Stealth[length=2.5mm]}] (yt+1.east) -- (bot.west);
        \draw[-{Stealth[length=2.5mm]}] (pasttop.east) -- (xt-2.west);
        \draw[-{Stealth[length=2.5mm]}] (pastbot.east) -- (yt-2.west);
        \draw[-{Stealth[length=2.5mm]}] (pasttop.south) -- (yt-2.north);
        \draw[{Stealth[length=2.5mm]}-{Stealth[length=2.5mm]}] (yt-1.north) --  (xt.south);
        \draw[{Stealth[length=2.5mm]}-{Stealth[length=2.5mm]}] (yt-2.north) --  (xt-1.south);
        \draw[{Stealth[length=2.5mm]}-{Stealth[length=2.5mm]}] (yt.north) --  (xt+1.south);
        \draw[{Stealth[length=2.5mm]}-{Stealth[length=2.5mm]}] (yt+1.north) --  (top.south);
        \draw[{Stealth[length=2.5mm]}-{Stealth[length=2.5mm]}] (pastbot.north) --  (xt-2.south);
    \end{tikzpicture}
    \caption[Full-time marginalized  ADMG  of the process in \eqref{eq:ExampleVarmaDAG}.]{ Full-time marginalized  ADMG $\mathcal{G}_{S}$ of the time series  defined in Equation \eqref{eq:ExampleVarmaDAG}.}
    \label{fig:ExampleVarmaADMG}
    \end{figure}
    
    With the aim of identifying the total causal effect of $X_{t-1}$ on $Y_t$, we could try and consider $Y=Y_t, \mathcal{X}=\{X_{t-1}\},\mathcal{B}=\emptyset$ and $\mathcal{I}=\{X_{t-2}\}$. However, in this setting condition \textit{(1')} is not satisfied  due to the open path $X_{t-2}\to Y_{t-1}\to Y_t$. In order to circumvent this issue, we extend $\mathcal{X}$ to $\mathcal{X}=\{X_{t-1}, Y_{t-1}\}$, and then take $\beta_1$ as the total causal effect from $X_{t-1}$ on $Y_t$. This is correct owing to the fact that no directed path from $X_{t-1}$ to $Y_t$ crosses $Y_{t-1}$. Yet at this point another issue arises: We have $\dim(\mathcal{X})=2$ but $\dim(\mathcal{I})=1$ meaning that, as explained in Remark \ref{remark:IVDimensionality}, the total effect will be under-identified by the {IV} regression so far defined. To solve this problem, we extend the set of instruments to $\mathcal{I}=\{X_{t-2},Y_{t-2}\}$. Condition \textit{(1')} now  holds. Furthermore, one can easily check that:
    $$\mathbb{C}\text{ov}(\mathcal{X},\mathcal{I})=\begin{pmatrix}
        17/24 & 137/108 \\ 53/54 & 377/972
    \end{pmatrix},$$ which has  full row rank of 2, and therefore condition \textit{(3')}  is also satisfied. Therefore we can use Theorem \ref{thm:IVregressionVARMA}  to compute the total causal effect of $\mathcal{X}$ on $Y=Y_t$. Indeed, one can check that $\mathbb{C}\text{ov}(Y,\mathcal{I})=(157/216, 1199/1944)^{\top}$ and thus: $$\beta \begin{pmatrix}
        17/24 & 137/108 \\ 53/54 & 377/972
    \end{pmatrix}=\begin{pmatrix} 157/216\\ 1199/1944 
    \end{pmatrix}\Rightarrow \beta=\left(\frac{1}{3},\frac{1}{2}\right),$$
    which is the correct total causal effect computed following Definition \ref{def:TotalCausalEffect}. Thus, we conclude that the total causal effect of $X_{t-1}$  on $Y_t$ is 1/3.
\end{example}
\section{Conclusions and future work}\label{sec:Conclusion}
In this work we have focused on the study of causal aspects and properties for time series under the structural assumption of a VARMA model with instantaneous effects. Allowing these effects in our model not only introduces greater complexity and generality. Thus, the first notable contribution presented in this work is the proof of the global Markov property for VAR processes with instantaneous effects (Theorem \ref{thm:GMPWith}), extending \cite[Thm. 1]{thams2022identifying}. In Section \ref{sec:GMPVarma} we  studied this global property for VARMA processes, again with instantaneous effects (Theorem \ref{thm:GmpVarmaWith}). We defined a full-time graph over the endogenous variables of the process, and proved that the GMP holds with respect to this infinite graph. 


In Section \ref{chap:faithfulness} we  studied the counterpart of the global Markov property: Faithfulness. As explained in said section, faithfulness holds almost surely for finite linear Gaussian systems under assumptions on the sampling of the process's coefficients. Taking advantage of the resemblance between VARMA models and linear Gaussian systems, a faithfulness statement is given in Theorem \ref{thm:FaithfulnessVARMAwith}.

At this point, with both the global Markov property and faithfulness at hand, we sink into the field of causal inference. Extending the work of \cite{thams2022identifying}, in Section \ref{chap:IV} we define the concept of total causal effect for VARMA processes with instantaneous effects. We developed an instrumental variable regression framework which allows to identify and consistently estimate said effects. The identifiability condition and the expression of this estimator can be found in Theorem \ref{thm:IVregressionVARMA}.

Overall, the most notable contribution of this work is the fact that we have added instantaneous effects and the MA part to existing techniques for VAR structures, which allows to consider  less trivial and wider families of processes. Nevertheless, this thesis poses more questions than answers. Regarding the global Markov property, a clear next step would be to extend it tp a broader class of processes, like  VARCH or V-GARCH structures, or even non-parametric formulations. This would get closer to the recent work of \cite{hochsprungglobal}. A discussion of these authors' work can be found in Section \ref{sec:Hochsprung}. 

The faithfulness study in Chapter \ref{chap:faithfulness} could also be expanded. The assumptions made in Theorem \ref{thm:FaithfulnessVARMAwith} are standard and analogous to those made in \cite[Thm. 3.2]{spirtes2001causation}. Thus, next steps would be similar to those for the GMP: Abandoning the linearity and/or the additive noise assumptions so as to expand the range of processes considered. Regarding the {IV} estimator of total causal effects defined in Theorem \ref{thm:IVregressionVARMA}, further of its properties could be investigated. While we have provided a limit in probability to the true causal effect, asymptotic distribution results are in order. Ideally, a central limit theorem could be proven. This would allow to easily compute confidence intervals for this estimator.  These, among others, might be the subject of future work.

\begin{credits}
\subsubsection{\ackname} I  thank Prof. Dr. Jonas Peters for posing the problems, the regular meetings and the helpful scientific discussions, as the contect of this work arises from my (unpublished) master thesis at ETH-Zurich. I remain thankful to the Barri\'e Foundation, of which I am a Fellow of the 2022 call, for their support throughout my graduate studies.

\subsubsection{\discintname}
The author has no competing interests to declare that are
relevant to the content of this article. 
\end{credits}
%
%
%
\bibliographystyle{splncs04}
\bibliography{myReferences}
%




\appendix
\section{Proofs}\label{sec:AppendixProofs}
\propExtension*
\begin{proof}
    $A\perp_{\mathcal{G}}C|B$  is equivalent to $A$ and $C$ being separated by $B$ in $\mathcal{G}^m$. Then $A$ and $C$ are in different connected components of $(\mathcal{G}^m)_{V\backslash B}$. Consider  $A^+$ the connected component of $(\mathcal{G}^m)_{V\backslash B}$ which contains $A$, and idem for $C$. Regarding  all other nodes in $(\mathcal{G}^m)_{V\backslash B}$ (if there were any), as they are in other connected components of $(\mathcal{G}^m)_{V\backslash B}$, they are separated from both $A^+$ and $C^+$ by $B$ in $\mathcal{G}^m$, or equivalently they are $d$-separated from both $A^+$ and $C^+$ by $B$ in $\mathcal{G}$. Therefore they can be arbitrarily added to $A^+$ or to $C^+$ while keeping these sets disjoint and $d$-separated by $B$. 
\end{proof}
\coroExtension*
\begin{proof}
    Consider a DAG $\mathcal{G}$ over nodes $V$. Let $A,B,C\subset V$ pairwise disjoint sets of nodes such that $V=AN_{\mathcal{G}}(A\cup B\cup C)$ and $A\perp_{\mathcal{G}}C|B$. Consider $A^+,C^+\subset V$ such that $A\subseteq A^+$, $C\subseteq C^+$, $A^+\cap C^+=\emptyset$, $A^+\perp_{\mathcal{G}}C^+|B$ and $V=A^+\cup B \cup C^+$. This $d$-separation statement is equivalent to $A^+$ and $C^+$ being separated by $B$ in $\mathcal{G}^m$, i.e. they are subsets of different connected components of $(\mathcal{G}^m)_{V\backslash B}$. As this graph only has two connected components and $V=A^+\cup B \cup C^+$, necessarily $A^+$ and $C^+$ are precisely the two connected components of $(\mathcal{G}^m)_{V\backslash B}$. This concludes the proof.
\end{proof}
\propSeparationADMG*
\begin{proof}
    This proof is a generalization of the proof of \cite[Prop. 3]{lauritzen1990independence}, which makes an analogous statement for the moralization of a DAG. We now prove sufficiency and necessity. 
    
    ``$\Leftarrow$'': Assume $B$ does not $m$-separate $A$ and $C$. Then there is a path $\mathfrak{p}$ $m$-connecting $A$ and $C$ given $B$ in $\mathcal{G}$. This path is contained in $AN_{\mathcal{G}}(A\cup B\cup C)$. Indeed, if a node $\gamma$ in $\mathfrak{p}$ is a collider, then $\gamma\in AN_{\mathcal{G}}(B)$. Otherwise,  one of the sides of $\mathfrak{p}$ exits $\gamma$. Assume w.l.o.g. that we have $\gamma\to\cdots A$. If all the edges in this path are directed edges towards $A$, then $\gamma\in AN_{\mathcal{G}}(A)$. Otherwise we encounter a collider in the path which is a descendant of $\gamma$  and ancestor of $B$, meaning $\gamma\in AN_{\mathcal{G}}(B)$. Furthermore, every node of $\mathfrak{p}$ which is in $B$ is necessarily a collider and therefore, when augmenting, an edge is created circumventing this node, meaning we create a path in $(\mathcal{G}_{AN_{\mathcal{G}}(A\cup B\cup C)})^{aug}$ between $A$ and $C$ which does not contain nodes in $B$. Therefore, $A$ and $C$ are not separated by $B$ in $(\mathcal{G}_{AN_{\mathcal{G}}(A\cup B\cup C)})^{aug}$. 

    ``$\Rightarrow$'': Assume $A$ and $C$ are not separated by $B$ in $(\mathcal{G}_{AN_{\mathcal{G}}(A\cup B\cup C)})^{aug}$. Then there is a path $\mathfrak{q}$  between $A$ and $C$ in $(\mathcal{G}_{AN_{\mathcal{G}}(A\cup B\cup C)})^{aug}$ circumventing $B$. This path has edges corresponding to edges in the original ADMG, and others originated by marrying collider connected nodes. We will now build a path $\mathfrak{p}$ in $\mathcal{G}$ $m$-connecting $A$ and $C$ given $B$. For every edge in $\mathfrak{q}$ corresponding to an edge in $\mathcal{G}$, we keep that edge and the nodes it connects in $\mathfrak{p}$. If an edge of $\mathfrak{q}$ is not in $\mathcal{G}$, then there is a collider path in $\mathcal{G}_{AN_{\mathcal{G}}(A\cup B\cup C)}$ between the nodes that edge connects, and we take that collider path in $\mathfrak{p}$. First, we can assume that every node in $\mathfrak{q}$ is a non-collider in $\mathfrak{p}$. Otherwise, there would be an edge in $(\mathcal{G}_{AN_{\mathcal{G}}(A\cup B\cup C)})^{aug}$ connecting the  nodes adjacent  to this collider in $\mathfrak{q}$, meaning we can take this edge in $\mathfrak{q}$ and skip the collider. Second, for all colliders in $\mathfrak{p}$ (which are not nodes of $\mathfrak{q}$), two things can happen. Let $\gamma$ be a collider in $\mathfrak{p}$:
    \begin{enumerate}
        \item If $\gamma$ has a descendant in $B$ it does not block $\mathfrak{p}$ given $B$,
        \item if $\gamma$ has no descendants in $B$, then it is an ancestor of $A$ or $C$, and all the nodes in such a descendant path are non-colliders not in $B$. Furthermore, when taking this descendant path back to $\gamma$ and then following $\mathfrak{p}$, $\gamma$ becomes a non-collider not in $B$.
    \end{enumerate}
    Now we will modify $\mathfrak{p}$ to obtain a $m$-connecting path.
    \begin{itemize}
        \item[(a)]  If $\mathfrak{p}$ has no colliders, or all colliders in $\mathfrak{p}$ are under i.), then $\mathfrak{p}$ is an open path between $A$ and $C$ given $B$ in $\mathcal{G}$.
        \item[(b)]  If there is a collider in $\mathfrak{p}$ under ii.), we again separate two cases:
        \begin{itemize}
            \item[(b.1)] All colliders under ii.) have descendants in $C$. Then we traverse $\mathfrak{p}$ starting in $A$, and once we reach the first of these colliders, we take the descendant path to $C$. This is an open path given $B$.
            \item[(b.2)] some colliders under ii.) have descendants in $A$, and others in $C$. Denote by $\delta$ the last collider under ii.) with descendants in $A$ we encounter when traversing $\mathfrak{p}$ from $A$ to $C$. We  take this ancestral path from $A$ to $\delta$, and then traverse $\mathfrak{p}$ towards $C$. If we find a collider under ii.), which necessarily has a descendant in $C$, we take this descendant path to $C$. Otherwise we follow the former path until $C$. This gives  an open path.
        \end{itemize}
    \end{itemize}
    This concludes the proof.    
\end{proof}
\propExtensionADMG*
\begin{proof}
    This proof is analogous to that of Proposition \ref{prop:Extension}, taking Proposition \ref{prop:SeparationADMG} into account.
\end{proof}
%

\propExtendedDAG*
\begin{proof}
    Denote $B=(I_d-A_0)^{-1}$. As $I_d-A_0$ is lower triangular with ones on the diagonal, so is $B$. Then the coefficients of the edges with lag $k\in\{1,\ldots,p\}$ are given by $BA_k$. For the sake of notation, denote $A=A_k$. Then the coefficient of the edge from $S_{t-k}^i$ to $S_t^i$ in $\mathcal{G}_{II}$ is given by: 
    \begin{equation}\label{eq:EdgeCoeffs}
        (BA)_{i,j}=\sum_{l=1}^dB_{i,l}A_{l,j}=A_{i,j}+\sum_{l=1}^{i-1}B_{i,l}A_{l,j}.
    \end{equation}
    For said edge to exist in $\mathcal{G}_{II}$, this coefficient has to be non-zero. For that we need at least one of the addends to be non-zero, meaning either $A_{i,j}\neq 0$ (the edge already exists in $\mathcal{G}_I$) or $B_{i,l}A_{l,j}\neq 0$ for some $l\in\{1,\ldots,i-1\}$. To understand the meaning of $B_{i,l}A_{l,j}\neq 0$ we need to give a graph interpretation to $B_{i,l}$. The exact expression of $B$ can be found in \cite[Thm. 2]{baliarsingh2015explicit}, from which we conclude that:
    \begin{equation}\label{eq:InverseMatrix}
        B_{i.j}=\begin{cases}
            1 & \text{ if }i=j\\
            0 & \text{ if } j>i\\
            ICE^{(d)}_{j\to i} & \text{ if } j<i
        \end{cases},
    \end{equation}
    where $ICE^{(d)}_{j\to i}$ is defined as follows: In the DAG of instantaneous effects with $d$ nodes, $ICE^{(d)}_{j\to i}$ is the sum over all directed paths from $S^j$ to $S^i$ of the product of the coefficients of the edges in those paths (this is  the total causal effect of $S_t^j$ on $S_t^i$ in the DAG of instantaneous effects as per Definition \ref{def:TotalCausalEffect}). The proof of Equation \eqref{eq:InverseMatrix} is done by induction of the dimension of the process $d$. For $d=2$ it holds, as:
    $$\begin{pmatrix}
        1&0\\-\alpha_{2,1}&1
    \end{pmatrix}^{-1}=\begin{pmatrix}
        1&0\\\alpha_{2,1}&1
    \end{pmatrix}=\begin{pmatrix}
        1&0\\ICE^{(2)}_{1\to 2}&1
    \end{pmatrix}.$$
    Now  assume that Equation \eqref{eq:InverseMatrix} holds for $d$, and let us prove it for $d+1$. Following \cite{baliarsingh2015explicit} we know that $B$ is lower triangular with a diagonal of ones. Therefore, the equation $(I_{d+1}-A_0)B=I_{d+1}$ can be written out as:
    \begin{equation*}
        \begin{pmatrix}
            1 & 0 &   \dots  &  \dots &0\\
            -\alpha_{2,1} & 1 &\ddots &   &\vdots\\
            \vdots &  \ddots  & \ddots  &  0 &\vdots\\
            -\alpha_{d,1} &\dots& -\alpha_{d,d-1}&1&0\\
            -\alpha_{d+1,1} &\dots& -\alpha_{d+1,d-1}&-\alpha_{d+1,d}&1
        \end{pmatrix}
        \begin{pmatrix}
            1 & 0 &   \dots  &  \dots &0\\
            B_{2,1} & 1 &\ddots &   &\vdots\\
            \vdots &  \ddots  & \ddots  &  0 &\vdots\\
            B_{d,1} &\dots& B_{d,d-1}&1&0\\
            B_{d+1,1} &\dots& B_{d+1,d-1}&B_{d+1,d}&1
        \end{pmatrix}=I_{d+1},
    \end{equation*}
    where we see that the equations related to the first $d$ rows of $B$ do not involve any element from row $d+1$, i.e. block matrix multiplication tell us that the first $d$ rows of $B$ are the same in dimension $d$ than in dimension $d+1$. The induction hypothesis tells us then that $B_{i,j}=ICE^{(d)}_{j\to i}$ for $j\in\{1,\ldots,i-1\}$ and $i\in\{2,\ldots,d\}$. Recall that we assume that the instantaneous nodes are numbered by the instantaneous topological ordering. Thus  $ICE^{(d)}_{j\to i}=ICE^{(d+1)}_{j\to i}=B_{i,j}$ for all $j\in\{1,\ldots,i-1\}$ and $i\in\{2,\ldots,d\}$. To compute $B_{d+1,j}$ we use the equation arising from row $d+1$ in $I_{d+1}-A_0$ and column $j$ in $B$, which reads: 
    $$-\alpha_{d+1,j}-\sum_{k=j+1}^d\alpha_{d+1,k}B_{k,j}+B_{d+1,j}=0,$$
    what, using the induction hypothesis again,  leads to:
    $$    B_{d+1,j}\overset{I.H.}{=}\alpha_{d+1,j}+\sum_{k=j+1}^d\alpha_{d+1,k}ICE^{(d)}_{j\to k}=ICE^{(d+1)}_{j\to d+1},$$
    as $\alpha_{d+1,k}ICE^{(d)}_{j\to k}$ is the sum of the products of the coefficients of all paths which go from $j$ to $k$ and then use the edge $k \to (d+1)$. Hence this proves Equation \eqref{eq:InverseMatrix}.
    
    Returning to the interpretation of the addends in Equation \eqref{eq:EdgeCoeffs} we can now understand that $B_{i,l}A_{l,j}\neq 0$ means that $S_{t-k}^j$ is a parent of $S_t^l$ in $\mathcal{G}_I$ ($A_{l,j}\neq 0$) and $S_t^l$ is an instantaneous ancestor of $S_t^i$ (meaning there is a directed instantaneous path from $S_t^l$ to $S_t^i$). Hence, an edge can appear in $\mathcal{G}_{II}$ from $S_{t-k}^j$ to  $S_t^i$ if $S_{t-k}^j$ is a parent of an instantaneous ancestor of  $S_t^i$. 

    However, $\mathcal{G}_{II}$  might not be  exactly $\mathcal{G}^*$, but it always contains less edges than this DAG. $\mathcal{G}^*$ contains every non-instantaneous edge $\mathcal{G}_I$ has, and some others. Yet, some of the non-instantaneous edges of $\mathcal{G}_I$ might vanish in $\mathcal{G}_{II}$ due to cancellation of the sum in Equation \eqref{eq:EdgeCoeffs}. Therefore, $\mathcal{G}^*$ contains every non-instantaneous edge in $\mathcal{G}_{II}$, and both have no instantaneous edges. This concludes the proof. 
\end{proof}
\coroAncestors*
\begin{proof}
    By construction, $AN_{\mathcal{G}_{II}}(S_t^i)$ contains no contemporaneous nodes. Regarding non-contemporaneous nodes, if $S_{t-k}^j\in AN_{\mathcal{G}_{II}}(S_t^i)$, then either $S_{t-k}^j$ is a parents of $S_{t}^i$ in $\mathcal{G}_I$, or it is a parent of a contemporaneous ancestor of $S_{t}^i$ in $\mathcal{G}_I$. In any case    $S_{t-k}^j\in AN_{\mathcal{G}_{I}}(S_t^i)$. This concludes the proof. 
\end{proof}

\begin{remark}
    The two sets of ancestors mentioned in Proposition \ref{coro:Ancestors} might not be equal because, as explained in the proof of Proposition \ref{prop:ExtendedDAG}, some edges which are not contemporaneous  in $\mathcal{G}_I$ might not exist in $\mathcal{G}_{II}$, changing  the ancestral set.  
\end{remark}
\propVarianceInnovations*
\begin{proof}
    Recall that we can, w.l.o.g. take $A_0$ to be triangular inferior with zeroes on its diagonal, so that $B=(I_d-A_0)^{-1}$ is also triangular inferior with diagonal of ones. As $\delta_t=B\varepsilon_t$ we have that:
    $$\delta_t^i=\sum_{k=1}^dB_{i,k}\varepsilon_t^k=\varepsilon_t^i+\sum_{k=1}^{i-1}B_{i,k}\varepsilon_t^k,$$
    and, assuming w.l.o.g. that $i<j$:
    $$\delta_t^j=\sum_{k=1}^dB_{j,k}\varepsilon_t^k=\sum_{k=1}^{i}B_{j,k}\varepsilon_t^k+\underbrace{\sum_{k=i+1}^{j-1}B_{j,k}\varepsilon_t^k+\varepsilon_t^j}_{\text{the innovations here do not enter } \delta_t^i}.$$
    Now, as $S_t^i$ and $S_t^j$ have no common instantaneous ancestors, $B_{j,i}=0$ and thus $\varepsilon_t^i$ does not appear in $\delta_t^j$. Furthermore, for $l<i$, again as $S_t^i$ and $S_t^j$ have no common instantaneous ancestors, $B_{i,l}B_{j,l}=0$, and thus $\varepsilon_t^l$ does not enter both $\delta_t^i$ and $\delta_t^j$  simultaneously. As the components of $\varepsilon_t$ are jointly independent, we conclude that $\delta_t^i\indep\delta_t^j$. 
    
\end{proof}
\lemmaGeneralization*
\begin{proof}
    Consider the MA($\infty$) representation of $A$ and $B$ given by the distributional equivalent representation of the process without instantaneous effects. The MA($\infty$) representation of $A$ is a linear combination (with absolutely addable coefficients) of the  innovations $\delta_t^i$ associated with the variables in $AN_{\mathcal{G}_{II}}(A)$. Due to the hypothesis of disjoint ancestral sets in $\mathcal{G}_I$ and Corollary \ref{coro:Ancestors}, we know that $AN_{\mathcal{G}_{II}}(A)\cap AN_{\mathcal{G}_{II}}(B)=\emptyset$. This means that no innovation $\delta_t^i$  appears in both MA($\infty$) representations of $A$ and $B$. However, this innovations could  be correlated. Nevertheless, as $A$ and $B$ have no contemporaneous ancestors (as they have no common ancestors in $\mathcal{G}_{I}$), Proposition \ref{prop:VarianceInnovations} guarantees that any contemporaneous innovations which might appear in their MA($\infty$) expansions  are independent (otherwise there would be two contemporaneous ancestors with  a common contemporaneous ancestor, which would be a common ancestor itself, in contradiction with the ancestral sets being disjoint). Therefore, the innovations which enter the MA($\infty$) representations of $A$ and $B$ are all independent, and thus $A\indep B$.
\end{proof}
\thmGMPWith*
\begin{proof}
    We first present here the proof of Theorem \ref{thm:GMPWithout}, which covers the case without instantaneous effects.     This proof is nothing but the one originally given in \cite[Thm. 1]{thams2022identifying}, where some points have been explained in greater depth. Given a time series (sub)graph $\mathcal{G}$ over nodes $V$ we define: $$ V_{[s,t]}\defeq\{S_v^i\colon i \in \{1,\ldots,d\} , s\leq v\leq t, S_v^i\in V\}.$$ Consider $s_0,t_0$ the largest and smallest time points respectively such that $A\cup B\cup C\subseteq V_{[s_0,t_0]}$, and let $q\in\mathbb{N}_{>p}$ be such that if two nodes in $V_{[s_0,t_0]}$ are $d$-connected (with empty conditioning set) in $\mathcal{G}_{full}$, then there is a $d$-connecting path in $\mathcal{G}_{[s_0-q,t_0]}$. Note that as $q>p$, $\mathcal{G}_{[s_0-q,s_0-1]}$ contains enough information to fully describe the time dependence of the process. Now consider the set: 
    \begin{equation*}
        \begin{split}
            \mathcal{A}\defeq\{n\in\mathbb{N}\colon  & \text{ for all graphs } \mathcal{G}^* \text{ over nodes } V^*\text{ with time indices between}\\& \quad\text{ $s_0-q$ and $t_0$ such that all edges in }\mathcal{G}^* \text{ point ``forward} \\&\quad\text{in time'', meaning for all }k\in\mathbb{N} \text{ and } i,j\in\{1,\ldots,d\}\text{ there is}\\&\quad\text{ no edge }S_t^i\to S_{t-k}^j, \text{ and }\mathcal{G}^* \text{ is a subgraph of  the }\\&\quad\text{ full-time graph } \mathcal{G}_{full}\text{ induced by } \mathcal{G}^0\defeq \mathcal{G}^*_{[s_0-q,s_0-1]} \text{ and }\\&\quad\ AN_{\mathcal{G}_{full}}(V^+)_{[s_0,t_0]}=V^+, \text{ where } V^+\defeq V^*_{[s_0,t_0]}\text{ and } |V^+|=n, \\&\text{ for all VAR($p$) processes with structure specified by } \mathcal{G}^0,\\&\text{ for all } A^*,B,C^*\subseteq V^* \text{ such that } V^+=A^+\cup B \cup C^+ \text{ where } \\&\quad\ A^+\defeq A^*_{[s_0,t_0]}, C^+\defeq C^*_{[s_0,t_0]}  \text{ and } A^*=A^+\cup (PA_{\mathcal{G}^*}(A^+)\cap V^0)\\&\quad  \text{ and }  C^*=C^+\cup (PA_{\mathcal{G}^*}(C^+)\cap V^0),\text{ with }V^0\text{ the nodes of }\mathcal{G}^0\\&\text{ we have: }\\&\quad\ A^*\perp_{\mathcal{G}^*}C^*|B\Rightarrow A^*\indep C^*|B      \}.
        \end{split}
    \end{equation*}
    The proof first shows that $\mathcal{A}=\mathbb{N}$ suffices to prove the theorem, and secondly shows that indeed $\mathcal{A}=\mathbb{N}$ by induction on $n$. Assume then that $\mathcal{A}=\mathbb{N}$. Consider $V^0\defeq(V_{full})_{[s_0-q,s_0-1]}$, $V^+\defeq AN_{\mathcal{G}_{full}}(A\cup B\cup C)_{[s_0,t_0]}$, $V^*\defeq V^0\cup V^+$ and $\mathcal{G}^*$ the restriction of $\mathcal{G}_{full}$ to $V^*$. Note how $A\perp_{\mathcal{G}_{full}}C|B\Rightarrow A\perp_{\mathcal{G}^*}C|B$. However, $A,B,C$ might not satisfy the conditions required for $A^*,B,C^*$. We need to expand $A$ and $C$ appropriately: 
    \begin{enumerate}
        \item If $V^+\neq A\cup B \cup C$, then we expand $A$ and $C$ to disjoint sets $A^+$ and $C^+$ respectively, as per Proposition \ref{prop:Extension}, such that $V^+=A^+\cup B\cup C^+$ and still $A^+\perp_{\mathcal{G}^*}C^+|B$.
        \item Consider $A^*\defeq A^+\cup (PA_{\mathcal{G}^*}(A^+)\cap V^0) $ and $C^*\defeq C^+\cup (PA_{\mathcal{G}^*}(C^+)\cap V^0)$. Note how $A^+=A^*_{[s_0,t_0]}=A^*\cap V^0$ (idem for $C^+$) as $V^+\cap V^0=\emptyset$.  Also, $A^*\cap C^*=\emptyset$, as otherwise, given than $A^+\cap C^+=\emptyset$, there would be a common parent of $A^+$ and $C^+$ not in $B$, violating $d$-separation. 
    \end{enumerate}
    It also holds  that  $A^*\perp_{\mathcal{G}^*}C^*|B$. Otherwise, if there is an open path from $a\in A^*$ to $c\in C^*$ given $B$, then there is an open path from $A^+$ (either $a$ or its descendant in $A^+$) to $C^+$ given $B$, violating $d$-separation. Now that we are under all conditions of $\mathcal{A}$, the hypothesis $\mathcal{A}=\mathbb{N}$ implies that $A^*\indep C^*|B$, what implies $A\indep C|B$ as $A\subset A^+\subset A^*$, and idem for $C$.

    Now it rests to show that $\mathcal{A}=\mathbb{N}$. This is done by induction on $n=|V^+|$. Recall that $V^+=A^+\cup B\cup C^+$ pairwise disjoint union.  Firstly note that $A^+=\emptyset\Rightarrow A^*=\emptyset$ (idem for $C^+$). Thus for $n=1$ and $n=2$ the only non-trivial case is $A^+=\{a\}, C^+=\{c\},B=\emptyset$. As  $A^*\perp_{\mathcal{G}^*}C^*$, $A^*$ and $C^*$ have no common ancestors in $\mathcal{G}^*$, but as $\mathcal{G}^*$ contains more than $p$ time steps, the repetitive structure of $\mathcal{G}_{full}$ implies that they have no common ancestors whatsoever. By iterating the structural equations (MA($\infty$) representation)  one concludes that $A^*\indep C^*$.

    It remains to prove that $n\in\mathcal{A}\Rightarrow n+1\in\mathcal{A}$. Consider $\mathcal{G}^*,\mathcal{G}^0,\mathcal{G}^+,A^+,A^*,B,$ $C^+,C^*$ as per the definition of $\mathcal{A}$ with $|V^+|=n+1$. The basic idea is to remove a node from $V^+$, consider the independence statements granted by the induction hypothesis, and then combine them to conclude the desired independence. Consider then $\lambda\in V^+$ a sink node in $\mathcal{G}^*$. As $V^+=A^+\cup B\cup C^+$ disjoint union, there are three cases to distinguish:
    \begin{itemize}
        \item Assume $\lambda\in A^+$. Then $C^*\cap PA_{\mathcal{G}_{full}}(\lambda)=\emptyset$ (as otherwise $d$-separation would be violated), what implies $PA_{\mathcal{G}_{full}}(\lambda)=PA_{\mathcal{G}^*}(\lambda)\subseteq (A^*\backslash\{\lambda\})\cup B$. Then there are coefficients $\gamma\in\mathbb{R}^{|PA_{\mathcal{G}^*}(\lambda)|}$ such that the structural equation for $\lambda$ is $\lambda=\gamma^{\top}PA_{\mathcal{G}^*}(\lambda)+\varepsilon^\lambda$. Thanks to independence of all innovations and the fact that $\lambda$ is a sink node (it does not appear in any other structural equation of nodes in $\mathcal{G}^*$), we conclude that:
        \begin{equation}\label{eq:IndependenceInA}
            \lambda\indep (A^*\cup B \cup C^*)\backslash(\{\lambda\}\cup PA_{\mathcal{G}^*}(\lambda))|PA_{\mathcal{G}^*}(\lambda),
        \end{equation} 
        and by the weak union property $\lambda\indep C^*|(A^*\cup B)\backslash\{\lambda\}$. If $A^*=\{\lambda\}$, then this already implies $A^*\indep C^*|B$. Otherwise, $A^*\perp_{\mathcal{G}^*}C^*|B$ implies $A^*\backslash\{\lambda\}\perp_{\mathcal{G}^*}C^*|B$, which is equivalent (by Proposition \ref{prop:SeparationMoralized}) to $A^*\backslash\{\lambda\}$ being separated from $C^*$ by $B$ in $(\mathcal{G}^*)^m$.  Then this separation also holds in $(\mathcal{G}^*_{V^*\backslash\{\lambda\}})^m$, as this graph contains no more edges. Again by Proposition \ref{prop:SeparationMoralized} this separation is equivalent to $A^*\backslash\{\lambda\}\perp_{\mathcal{G}^*_{V^*\backslash\{\lambda\}}}C^*|B$. As $|V^+\backslash\{\lambda\}|=n$, the induction hypothesis yields $A^*\backslash\{\lambda\}\indep C^*|B$. Combining these last two conditional independence statements by the contraction property, one concludes that $A^*\indep C^*|B$.
        \item Assume $\lambda\in C^+$. This scenario is analogous to the $\lambda\in A^+$ case.
        \item Assume $\lambda\in B$. The $d$-separation hypothesis is equivalent to $A^*$ and $C^*$ being separated by $B$ in $(\mathcal{G}^*)^m$, what implies they are also separated by $B\backslash\{\lambda\}$ in $(\mathcal{G}^*_{V^*\backslash\{\lambda\}})^m$ (as this graph contains no more edges than $(\mathcal{G}^*)^m$), what is equivalent to $A^*\perp_{\mathcal{G}^*_{V^*\backslash\{\lambda\}}}C^*|B\backslash\{\lambda\}$. As $|V^+\backslash\{\lambda\}|=n$, the induction hypothesis implies $A^*\indep C^*|B\backslash\{\lambda\}$. To achieve the desired independence claim, we need a second independence statement. We consider two cases:
        \begin{itemize}
            \item[(a)] Assume that:
            \begin{equation}\label{eq:ConditionsA}
                \begin{cases}
                    PA_{\mathcal{G}^+}(\lambda)\cap A^*\neq\emptyset\\\text{ or }\\AN_{\mathcal{G^*}}(PA_{\mathcal{G}^*}(\lambda)_{[s_0-q,s_0-1]})\cap AN_{\mathcal{G^*}}(PA_{\mathcal{G}^*}(A^*)_{[s_0-q,s_0-1]})\neq\emptyset
                \end{cases}.
            \end{equation}
            Then it holds that: 
            \begin{equation}\label{eq:ConsequenceA}
                \begin{cases}
                    PA_{\mathcal{G}^+}(\lambda)\cap C^*=\emptyset\\\text{ and }\\ AN_{\mathcal{G^*}}(PA_{\mathcal{G}^*}(\lambda)_{[s_0-q,s_0-1]})\cap AN_{\mathcal{G^*}}(PA_{\mathcal{G}^*}(C^*)_{[s_0-q,s_0-1]})=\emptyset
                \end{cases}.
            \end{equation}
            Indeed, if the first statement in \eqref{eq:ConsequenceA} was false, there would be a $d$-connecting path between $A^*$ and $C^*$ given $B$: First from the element in $PA_{\mathcal{G}^+}(\lambda)\cap C^*$ to $\lambda\ (\in B)$ and then either to the element in $PA_{\mathcal{G}^+}(\lambda)\cap A^*$ or to the common ancestor of  $PA_{\mathcal{G}^*}(\lambda)_{[s_0-q,s_0-1]}$ and $PA_{\mathcal{G}^*}(A^*)_{[s_0-q,s_0-1]}$ and then to $A^*$. Note how all elements in these paths are non-colliders not in $B$, except $\lambda\in B$. If the second statement in \eqref{eq:ConsequenceA} was false, then we can take a path from $C^*$ to $\lambda$ via the common ancestor of $PA_{\mathcal{G}^*}(\lambda)_{[s_0-q,s_0-1]}$ and $PA_{\mathcal{G}^*}(C^*)_{[s_0-q,s_0-1]}$, and then use the same path from $\lambda$ to $A^*$ as before, building again a $d$-connecting path given $B$. 

            The conditions given in Equation \eqref{eq:ConsequenceA} imply that: 
            \begin{equation}\label{eq:IndependenceInB}
                \lambda\indep C^*|A^*\cup (B\backslash\{\lambda\}).
            \end{equation}
            This owes to the fact that, as $PA_{\mathcal{G}^+}(\lambda)\subseteq A^*\cup (B\backslash\{\lambda\}) $, in the structural equation of $\lambda|A^*\cup (B\backslash\{\lambda\})$ we can replace $PA_{\mathcal{G}^*}(\lambda)_{[s_0-q,s_0-1]}$ by their MA($\infty$) representation. For $C^*$, their structural equations are iterated except for the nodes of $B\backslash\{\lambda\}$ and $PA_{\mathcal{G}^*}(C^*)_{[s_0-q,s_0-1]}$ (variables in $A^*$ cannot occur as then there would be a directed path from $A^*$ to $C^*$), followed by then substituting $PA_{\mathcal{G}^*}(C^*)_{[s_0-q,s_0-1]}$ by their MA($\infty$) representation. Note how, due to the lack of common ancestors, these two MA($\infty$) expansions involve no common innovations. Thus, as all the error terms are independent, we can conclude that $\lambda\indep C^*|A^*\cup (B\backslash\{\lambda\})$. Together with  $A^*\indep C^*|B\backslash\{\lambda\}$, the contraction property followed by the weak-union property imply $C^*\indep A^*|B$.
            \item[(b)] Now assume the complement of Equation \eqref{eq:ConditionsA}:
            \begin{equation}\label{eq:ConditionsB}
                \begin{cases}
                    PA_{\mathcal{G}^+}(\lambda)\cap A^*=\emptyset\\\text{ and }\\ AN_{\mathcal{G^*}}(PA_{\mathcal{G}^*}(\lambda)_{[s_0-q,s_0-1]})\cap AN_{\mathcal{G^*}}(PA_{\mathcal{G}^*}(A^*)_{[s_0-q,s_0-1]})=\emptyset
                \end{cases},
            \end{equation}
            which are precisely the conditions in Equation \eqref{eq:ConsequenceA} for $A^*$ instead than for $C^*$. Identical reasoning as in (a) leads to $A^*\indep C^*|B$.
        \end{itemize}
    \end{itemize}
    Therefore, $n+1\in\mathcal{A}$, what means $\mathcal{A}=\mathbb{N}$. This concludes the proof of  the case without instantaneous effects (Theorem \ref{thm:GMPWithout})

    We now will extend this proof to account for the possibility of instantaneous effects. We realize that only three aspects of the proof of Theorem \ref{thm:GMPWithout} need altering in the case of instantaneous effects. First is the case $n=2$ in the induction proof. As $A^*$ and $C^*$ have no common ancestors in $\mathcal{G}_I$, Lemma \ref{lemma:Generalization} implies $A^*\indep C^*$. 
    
    Second is the proof of Equation \eqref{eq:IndependenceInA}. In the representation of the process with instantaneous effects as given in Definition \ref{def:VARWith}, innovations are independent (both contemporaneous and not). Therefore, after conditioning on $PA_{\mathcal{G}_I^*}(\lambda)$, the only randomness left in $\lambda$ comes from $\varepsilon^\lambda$. Thus, the only way $\lambda$ could not be independent of a node in $\Delta=(C^*\cup B\cup A^*)\backslash (\{\lambda\}\cup PA_{\mathcal{G}_I^*}(\lambda))$ given $PA_{\mathcal{G}_I^*}(\lambda)$ is if $\varepsilon^\lambda$ appeared in an structural equation of a node of $\Delta$ when iterating these equations over $\Delta$. However, this never happens as $\lambda$ is a sink node in $\mathcal{G}_I^*$ (it appears in no structural equation of nodes of $\mathcal{G}_I^*$ except its own). Therefore, Equation \eqref{eq:IndependenceInA} still holds in the case with instantaneous effects. 

    Lastly, the other independence statement which needs revisiting is Equation \eqref{eq:IndependenceInB}. Note how the conditions in \eqref{eq:ConditionsA} still imply \eqref{eq:ConsequenceA} even with instantaneous effects. Then, Equation \eqref{eq:ConsequenceA} together with Lemma \ref{lemma:Generalization} imply that $PA_{\mathcal{G}_I^*}(\lambda)_{[s_0-q,s_0-1]}$ and $PA_{\mathcal{G}_I^*}(C^*)_{[s_0-q,s_0-1]}$ are independent. Therefore, Equation \eqref{eq:IndependenceInB} still holds in the context of instantaneous effects. This concludes the proof, as any other step in the proof of Theorem \ref{thm:GMPWithout} generalizes directly to the case with instantaneous effects.  
    
\end{proof} 
\thmGmpVarmaWith*
\begin{proof}
    Take the VARMA($p,q$) process with instantaneous effects: 
    $$(I):\quad S_t=A_0S_t+A_1S_{t-1}+\ldots +A_pS_{t-p}+\varepsilon_t+B_1\varepsilon_{t-1}+\ldots B_q\varepsilon_{t-q},$$ where we assume, w.l.o.g., that $A_0$ is triangular inferior. We can rewrite this process as a $2d$-dimensional VAR($l$) process, with $l=\max(p,q)$ as:
    $$(II):\quad Q_t=\begin{pmatrix}  S_t \\ \varepsilon_t  \end{pmatrix} =
    \underbrace{\begin{pmatrix} A_0 & 0 \\ 0 & 0 \end{pmatrix}}_{C_0} Q_t + 
    \underbrace{\begin{pmatrix} A_1 & B_1 \\ 0 & 0 \end{pmatrix}}_{C_1} Q_{t-1}+\cdots+\underbrace{\begin{pmatrix}      A_l & B_l \\ 0 & 0 \end{pmatrix}}_{C_l} Q_{t-l} + \eta_t, $$
    under the convention that $A_i=0$ for $i>p$ and $B_i=0$ for $i>q$. Furthermore, $\eta_t$ is an i.i.d. process with jointly independent components such that:
    $$\begin{cases}
        \mathbb{P}(\eta_t^i=0)=1 \text{ for all }i\in\{1,\ldots,d\}\\
        \eta_t^i\overset{d}{=}\varepsilon_t^{i-d} \text{ for all }i\in\{d+1,\ldots,2d\}
    \end{cases}.$$
    In this second equation we understand the distribution of $\varepsilon_t$ as the one in the definition of the VARMA process $(I)$. Therefore, $\eta_t$ is an i.i.d. process with finite second moments and jointly independent component. Furthermore, thanks to the conditions in the definition of the VARMA process $(I)$, it holds that:
    \begin{align*}
        \det ((I_{2d}-C_0)\lambda^l-C_1\lambda^{l-1}-\ldots -C_l)&=\lambda^{ld}\det ((I_d-A_0)\lambda^l-A_1\lambda^{l-1}-\ldots -A_l)\\&=0\Rightarrow |\lambda|<1.
    \end{align*}
    We conclude  that the VAR process $(II)$ satisfies the  conditions of Theorem \ref{thm:GMPWith}, meaning the global Markov property holds. Two facts must be highlighted at this point. First, the stationary distribution of the VARMA process $(I)$ and its innovations is  the stationary  distribution of $Q$: $\mathbb{P}_{Q}=\mathbb{P}_{(S,\varepsilon)}$. Second, the full-time DAG of the VAR process $\mathcal{G}^{(II)}$ is  the (not marginalized) full-time DAG of the VARMA process $\mathcal{G}^{(I)}_{full}$ (we add the superscript $(I)$ for clarity). 
    
    Consider then  $A,B,C$ finite pairwise disjoint sets of nodes of the full-time marginalized ADMG of the VARMA process $\mathcal{G}^{(I)}_{S}$  such that $A$ and $C$ are $d$-separated by $B$ in $\mathcal{G}^{(I)}_{S}$. By Proposition \ref{prop:SeparationEquivLatentProjection}, these nodes are also separated in the full-time DAG of the process $\mathcal{G}^{(I)}_{full}$.  Attending to the previous remarks, $A$ and $C$ are $d$-separated by $B$ in the full-time DAG  $\mathcal{G}^{(II)}$ and thanks to the global Markov property  (Theorem \ref{thm:GMPWith}), $A\indep C|B$. This concludes the proof. 
\end{proof}

\thmFaithfulnessVARWith*
\begin{proof}
    We follow the ideas of \cite[Thm. 3.2]{spirtes2001causation}. Firstly, note that $0\in\mathcal{D}$, and due to the continuity of the solutions of an algebraic equation w.r.t. its coefficients, $0\in\overset{\circ}{\mathcal{D}}$, meaning $\overset{\circ}{\mathcal{D}}$ is open and non-empty. Furthermore, given $A,B,C$ finite pairwise disjoint subsets of nodes  of the full-time DAG of the process, violation of $d$-separation is equivalent to a node  $a\in A$ not being $d$-separated from a node $c\in C$ given $B$ in $\mathcal{G}_{full}$. Thus, it is sufficient to consider single nodes $X,Y$ and $B$ a finite set of nodes disjoint of $\{X,Y\}$. 

    The first step is to compute the stationary distribution of $(X,Y)$ conditional on $B$. To do so we use the distribution-equivalent representation of the process without instantaneous effects: 
    \begin{equation}\label{eq:VarmaWithEquivalentWithout}
        S_t=CA_1S_{t-1}+\ldots + CA_pS_{t-p}+\delta_t,
    \end{equation}
    where $C=(I_d-A_0)^{-1}$ and $\delta_t\overset{\text{i.i.d.}}{\sim}\mathcal{N}(0,C\Gamma C^\top)$. 
    The stationarity and additive noise  of this equivalent representations allow us to compute conditional MA($\infty$)  representations of the form:
    \begin{equation}\label{eq:ConditionalMA}
         X|B\overset{d}{=}\gamma_X^\top B + \sum_{t\in\mathbb{Z}}\sum_{i=1}^d  \alpha_t^i\delta_t^i,\qquad        Y|B\overset{d}{=}\gamma_Y^\top B + \sum_{t\in\mathbb{Z}}\sum_{i=1}^d  \beta_t^i\delta_t^i,
    \end{equation}
    with $\sum_{t\in\mathbb{Z}}\sum_{i=1}^d|\alpha_i|$ and $\sum_{t\in\mathbb{Z}}\sum_{i=1}^d|\beta_i|$ finite. As this representation is obtained  by iterating the structural equations of the process except over the nodes of $B$ in the representation without instantaneous effects, we know that   $\alpha_t^i$ is the sum of the product of the coefficients associated to all directed paths from $S_t^i$ to $X$ which do not intersect $B$ in the full-time  DAG of the process without instantaneous effects (idem for $\beta_i$ and $Y$). Recall from the proof of Proposition \ref{prop:ExtendedDAG}  that the entries of $C$ are polynomials on vech($A_0$). Hence the coefficients of the distribution-equivalent representation without instantaneous effects are polynomials on $(\text{vech}(A_0),\text{vec}(A_1|\ldots|A_p))$, and then so are $\alpha_t^i$ and $\beta_t^i$. 

    Due to the Gaussianity and independence of the innovations, the stationarity of the process, Levy's Theorem, and the representation given in Equation  \eqref{eq:ConditionalMA}, we know that: 
    \begin{equation}\label{eq:StationaryConditionalDistribution}
        (X,Y)|B\sim\mathcal{N}\left(\begin{pmatrix}
        \gamma_X^\top\\ \gamma_Y^\top   \end{pmatrix}B, \sum_{t\in\mathbb{Z}}\sum_{i,j,k=1}^d
        \begin{pmatrix}
              \alpha_t^i\alpha_t^jC_{i,k}C_{j,k}\sigma_k^2 & -\\
             \alpha_t^i\beta_t^jC_{i,k}C_{j,k}\sigma_k^2 &  \beta_t^i\beta_t^jC_{i,k}C_{j,k}\sigma_k^2
        \end{pmatrix}\right).
    \end{equation}
    Thanks to the finite number of variances and non-zero components of $C$, and the absolute addibility of the coefficients of the conditional MA($\infty$) representations, the sums in the variances and covariance in Equation \eqref{eq:StationaryConditionalDistribution} are finite. The normality of $(X,Y)|B$ implies that $$X\indep Y|B  \Leftrightarrow \mathbb{C}\text{ov}(X,Y|B)=0\Leftrightarrow \sum_{t\in\mathbb{Z}}\sum_{i,j,k=1}^d \alpha_t^i\beta_t^jC_{i,k}C_{j,k}\sigma_k^2=0. $$ Thus, a violation of faithfulness happens when $X\not\perp_{\mathcal{G}_{full}}Y|B$ but  $\mathbb{C}\text{ov}(X,Y|B)=0$. Owing to the joint normality and stationarity of the process, \cite[Thm. 3.5]{spirtes2001causation} tells us that $X\perp_{\mathcal{G}_{full}}Y|B$ is equivalent to $\mathbb{C}\text{ov}(X,Y|B)=0$. Hence, if $X\not\perp_{\mathcal{G}_{full}}Y|B$, then $\mathbb{C}\text{ov}(X,Y|B)= \sum_{t\in\mathbb{Z}}\sum_{i,j,k=1}^d \alpha_t^i\beta_t^jC_{i,k}C_{j,k}\sigma_k^2$ is not trivially zero. Consider now the map $f\colon \overset{\circ}{\mathcal{D}}\times (0,\infty)^d \mapsto \sum_{t\in\mathbb{Z}}\sum_{i,j,k=1}^d \alpha_t^i\beta_t^jC_{i,k}C_{j,k}\sigma_k^2$. This map is not trivially zero. The absolute convergence of the covariance $f$ represents implies that $f$ is a real analytic function \cite[Thm 2.2.5]{krantz2002primer}. As $\overset{\circ}{\mathcal{D}}\times (0,\infty)^d\subset\mathbb{R}^{dp^2+d(d-1)/2+d}$ is open and connected and $f$ is real analytic, the set of its zeroes has Lebesgue measure zero (see  \cite{mityagin2015zero}). Therefore, the set of coefficients generating a faithfulness violation has Lebesgue measure zero, and therefore is a $\mathbb{P}$-null set. This concludes the proof.
\end{proof}
\thmFaithfulnessVARMAWith*
\begin{proof}
    For this proof, instead of working in the marginalized full-time ADMG $\mathcal{G}_S$, we will always consider the full-time (not marginalized) DAG of the process $\mathcal{G}_{full}$. Recall that by Proposition \ref{prop:SeparationEquivLatentProjection}, $m$-separation in $\mathcal{G}_S$ is equivalent to $d$-separation in $\mathcal{G}_{full}$ regarding endogenous variables. 

    The first steps are identical to the proof of Theorem \ref{thm:FaithfulnessVARWith}: By the same arguments, $\overset{\circ}{\mathcal{D}}$ is open and non-empty. Furthermore, for our purposes it suffices to consider single endogenous nodes $X,Y$ and a finite set of endogenous nodes $B$ disjoint of $\{X,Y\}$. 

    We begin by computing the stationary distribution of $(X,Y)|B$. To do so, given the VARMA process with instantaneous effects $$S_t=A_0S_t+A_1S_{t-1}+\ldots +A_pS_{t-p}+\varepsilon_t+B_1\varepsilon_{t-1}+\ldots B_q\varepsilon_{t-q},$$ we express it as a distributional-equivalent VARMA process of the same order and dimension without instantaneous effects and correlated innovations: 
    $$S_t=CA_1S_{t-1}+\ldots + CA_pS_{t-p}+\delta_t+CB_1C^{-1}\delta_{t-1}+\ldots +CB_qC^{-1}\delta_{t-q},$$
    using the same notation as in Equation \eqref{eq:VarmaWithEquivalentWithout}. The stationarity and additivity  of the process allow us to compute conditional MA($\infty$) representations analogous to those in Equation \eqref{eq:ConditionalMA}. Thus, the stationary distribution of $(X,Y)|B$ is analogous to the one given in Equation \eqref{eq:StationaryConditionalDistribution}, where the covariances are polynomials on the coefficients of the process $(\text{vech}(A_0),\text{vec}(A_1|\ldots|A_p|B_1|\ldots|B_q),\sigma_1^2,\ldots,\sigma_d^2)$.  Thus, a violation of faithfulness happens when $X\not\perp_{\mathcal{G}_{full}}Y|B$ but  $\mathbb{C}\text{ov}(X,Y|B)=0$. 
    
    The next step of the proof is to reason that $X\not\perp_{\mathcal{G}_{full}}Y|B$ implies that $\mathbb{C}\text{ov}(X,Y|B)$ is not trivially  zero. To be able to apply the results from \cite{spirtes2001causation}, we rewrite the VARMA($p,q$) process with instantaneous effects as a VAR($l$) process with instantaneous effects with $l=\max(p,q)$: 
    $$Q_t=\begin{pmatrix}  S_t \\ \varepsilon_t  \end{pmatrix} =
    \underbrace{\begin{pmatrix} A_0 & 0 \\ 0 & 0 \end{pmatrix}}_{C_0} Q_t + 
    \underbrace{\begin{pmatrix} A_1 & B_1 \\ 0 & 0 \end{pmatrix}}_{C_1} Q_{t-1}+\cdots+\underbrace{\begin{pmatrix}      A_l & B_l \\ 0 & 0 \end{pmatrix}}_{C_l} Q_{t-l} + \eta_t, $$
    under the convention that $A_i=0$ for $i>p$ and $B_i=0$ for $i>q$. Furthermore, $\eta_t$ is an i.i.d. process with jointly independent components such that:
    $$\begin{cases}
        \mathbb{P}(\eta_t^i=0)=1 \text{ for all }i\in\{1,\ldots,d\}\\
        \eta_t^i\overset{d}{=}\varepsilon_t^{i-d} \text{ for all }i\in\{d+1,\ldots,2d\}
    \end{cases}.$$
    In the proof of Theorem \ref{thm:FaithfulnessVARWith} the distribution condition required is that the stationary distribution of the process is multivariate normal. In the case of a VAR process, this is equivalent to the Gaussianity assumption imposed on the innovations of the process, what can be seen by conditioning on the observed parents of a node in the full-time DAG (without innovations). However, this VAR process does not have independent and normally distributed innovations. Yet, the stationary distribution of the process is indeed Gaussian, and the innovations are jointly independent. This suffices to apply Theorem 3.5 in \cite{spirtes2001causation}, meaning that if $X\not\perp_{\mathcal{G}_{full}}Y|B$ then $\mathbb{C}\text{ov}(X,Y|B)$ does not trivially vanish in $\mathbb{P}^{Q}$; but as $\mathbb{P}^Q_S=\mathbb{P}^S$, we conclude that   if $X\not\perp_{\mathcal{G}_{full}}Y|B$ then $\mathbb{C}\text{ov}(X,Y|B)$ does not trivially vanish in $\mathbb{P}^{S}$. The rest of the proof is analogous to the proof of Theorem \ref{thm:FaithfulnessVARWith}, as this non-zero covariance is an analytic function on the coefficients of the process. This concludes the proof.
\end{proof}

\begin{restatable}{lemma}{lemmaTranslationedgesVarmaWith}\label{lemma:TranslationedgesVarmaWith}
    Let $S$ be a VARMA($p,q$) process with instantaneous effects as per Definition \ref{def:VARMAWith}. Let $\mathcal{G}_I$ be its full-time DAG, and $\mathcal{G}_{I,S}$ be its marginalized full-time ADMG over the endogenous variables. Consider the distribution-equivalent process without instantaneous effects, and denote by $\mathcal{G}_{II}$ its full-time DAG and by $\mathcal{G}_{II,S}$ its marginalized full-time ADMG. Let $A$ be a finite set of nodes in the marginalized ADMG (meaning a finite collection of endogenous variables), and let $\varepsilon$ be an innovation of the process at a certain time point. Then, if the edge $\varepsilon\to AN_{\mathcal{G}_{II,S}}(A)$ exists in $\mathcal{G}_{II}$, then the edge $\varepsilon\to AN_{\mathcal{G}_{I,S}}(A)$ exists in $\mathcal{G}_I$.
\end{restatable}
\begin{proof}
    Firstly note that Corollary \ref{coro:Ancestors} still applies to the full-time DAGs of these VARMA processes, and hence $AN_{\mathcal{G}_{II,S}}(A)\subset AN_{\mathcal{G}_{I,S}}(A)$.  Now, let $Z\in AN_{\mathcal{G}_{II,S}}(A)$ be such that $\varepsilon\to Z$ exists in $\mathcal{G}_{II}$. If this edge exists in $\mathcal{G}_{I}$, then the proof is completed. Otherwise, due to the way $\mathcal{G}_{II}$ is constructed from $\mathcal{G}_{I}$, we know that  $\varepsilon$ has an edge into an instantaneous endogenous ancestor of $Z$. This means that in $\mathcal{G}_{I}$ we find the structure: $$\varepsilon\to \underbrace{W\to\dots\to}_{\text{all nodes and edges in }\mathcal{G}_{I,S}} Z \in AN_{\mathcal{G}_{I,S}}(A),$$ but then in $\mathcal{G}_{I}$ we have $\varepsilon\to W\in AN_{\mathcal{G}_{I,S}}(A)$, as we wanted to show.
\end{proof}

\begin{restatable}{proposition}{propDisjointMArepDirectVarmaWith}\label{prop:DisjointMArepDirectVarmaWith}
    Let $S$ be a VARMA($p,q$) process with instantaneous effects as per Definition \ref{def:VARMAWith}. Let $\mathcal{G}_I$ be its full-time DAG, and $\mathcal{G}_{I,S}$ be its  full-time marginalized ADMG over the endogenous variables. Consider  $A$ a finite set of nodes in the marginalized ADMG $\mathcal{G}_{I,S}$, and $\lambda$ a node in this ADMG not in $A$. Then, if $SP_{\mathcal{G}_{I,S}}(\lambda)\cap AN_{\mathcal{G}_{I,S}}(A)=\emptyset$,  no innovation which enters the structural equation of $\lambda$ directly (through the MA part of the process) can appear in the MA($\infty$) representation of $A$.
\end{restatable}
\begin{proof}
    Assume otherwise. Then there is an innovation $\varepsilon$ such that $\varepsilon\to\lambda$ exists in $\mathcal{G}_I$. Furthermore, if we denote by $\mathcal{G}_{II}$ the full-time DAG and by $\mathcal{G}_{II,S}$ the marginalized full-time ADMG of the distribution-equivalent process without instantaneous effects, we also have the edge $\varepsilon\to AN_{\mathcal{G}_{II,S}}(A)$ in $\mathcal{G}_{II}$. Due to Lemma \ref{lemma:TranslationedgesVarmaWith}, we know that in $\mathcal{G}_I$ we have the edge $\varepsilon\to AN_{\mathcal{G}_{I,S}}(A)$. These two structures in $\mathcal{G}_I$ marginalize to the bi-directed edge $\lambda\leftrightarrow AN_{\mathcal{G}_{I,S}}(A)$ in $\mathcal{G}_{I,S}$, what is a contradiction. This concludes the proof.
\end{proof}

\thmIVregressionVARMA*
\begin{proof}
     This proof follows the ideas of \cite[Thm. 5]{thams2022identifying}. We begin by proving \textit{(i)}. Due to the linearity and additive noise of the VARMA process, we can write the structural equation of $Y$ as a linear combination of its parents plus a linear combination of innovations: $$Y=\underbrace{\xi^\top PA_{\mathcal{G}_S}(Y)}_{ \text{AR part}}+\underbrace{\theta^{\top}\varepsilon^Y}_{\text{MA part}},$$ where $\varepsilon^Y$ is a vector of the innovations which enter the structural equation of $Y$ through its MA part, and $\xi,\theta$ are real vectors of the appropriate dimensions. We can recursively iterate the structural equations of the parents of $Y$ without replacing the variables in $\mathcal{X}\cup\mathcal{B}\cup ND_{\mathcal{G}_S}(\mathcal{X})$ until we reach for the first time a decomposition of the form $Y=\beta\mathcal{X}+\pi\mathcal{B}+\gamma R+\varepsilon$, where $R$ are all the variables which appear when iterating the structural equations of the parents of $Y$, which are not in $\mathcal{B}$ nor in any causal path from $\mathcal{X}$ to $Y$, meaning $R\cap (\mathcal{X}\cup\mathcal{B}\cup DE_{\mathcal{G}_S}(\mathcal{X}))=\emptyset$ (note how $R$ might contain descendants of $\mathcal{B}$), and $\varepsilon$ is the linear combination of accumulated innovations from the MA part of the structural equation of $Y$ and its iterated ancestors. Also note how the vector of coefficients in front of $\mathcal{X}$ is  the total causal effect $\beta$, as following Definition \ref{def:TotalCausalEffect} it is the sum of the products of the path coefficients along all directed paths from $\mathcal{X}$ to $Y$ which do not cross $\mathcal{X}$ (causal paths) in the full-time DAG of the process, which are not blocked by $\mathcal{B}$ by assumption \textit{(2)}. In order to prove \textit{(i)} we make two claims:
     \begin{itemize}
         \item[(a)] Any path from $\mathcal{I}$ to $R$ in $\mathcal{G}_S$ is blocked by $\mathcal{B}$, i.e. $\mathcal{I}\perp_{\mathcal{G}_S}R|\mathcal{B}$,
         \item[(b)] $\mathbb{E}[\mathbb{C}\text{ov}(\varepsilon,\mathcal{I}|\mathcal{B})]=0$.
     \end{itemize}
     Assuming  both of these statements, from (a) and the global Markov property for VARMA processes with instantaneous effects (Theorem \ref{thm:GmpVarmaWith}) we have that $\mathcal{I}\indep R|\mathcal{B}$, and hence $\mathbb{E}[\mathbb{C}\text{ov}(R,\mathcal{I}|\mathcal{B})]=0$. Thus, taking (b) into account and that trivially $\mathbb{E}[\mathbb{C}\text{ov}(\mathcal{B},\mathcal{I}|\mathcal{B})]=0$, we conclude that $\mathbb{E}[\mathbb{C}\text{ov}(Y-\beta\mathcal{X},\mathcal{I}|\mathcal{B})]=\mathbb{E}[\mathbb{C}\text{ov}(\pi\mathcal{B}+\gamma R+\varepsilon,\mathcal{I}|\mathcal{B})]=0$ as we wanted to show.

     To prove (a) we suppose that there exists $i\in\mathcal{I}, m\in R$ and a path $\mathfrak{p}:i-v_1-\cdots -v_n-m$ in $\mathcal{G}_S$ which is not blocked by $\mathcal{B}$ (where $v-w$ means there is an edge, directed or bi-directed, between $v$ and $w$ in $\mathcal{G}_S$). Therefore, all non-colliders in $\mathfrak{p}$ are not in $\mathcal{B}$ and all colliders in $\mathfrak{p}$ have a descendant in $\mathcal{B}$ in $\mathcal{G}_S$. To find a contradiction, we will construct a path in $\mathcal{G}_S$ from $\mathcal{I}$  to $Y$ violating assumption \textit{(1)}. To do so, we begin by noticing that: 
     \begin{itemize}
         \item $m\neq Y$ as it was found among its ancestors,
         \item $m\notin \mathcal{B}\cup\mathcal{X}$ by construction of $R$,
         \item $m\notin\mathcal{I}$: By construction of $R$, there is a directed path from $m$ to $Y$ not intersecting $\mathcal{B}\cup\mathcal{X}$ in $\mathcal{G}_S$. If $m\in\mathcal{I}$, this path would violate assumption \textit{(1)}.
     \end{itemize}
     Now concatenate the path $\mathfrak{p}$ from $i\in\mathcal{I}$ to $m\in R$ with the descendant path from $m\in R$ to $Y$. Denote this path by $\mathfrak{q}$. Note how both of the concatenated paths are not blocked by $\mathcal{B}$ (as the descendant path does not intersect $\mathcal{B}\cup\mathcal{X}$) and $m$ is not a collider in $\mathfrak{q}$. Thus, $\mathfrak{q}$  is not blocked by $\mathcal{B}$ in $\mathcal{G}_S$.


     Next we will show that this path is still not blocked by $\mathcal{B}$ when we remove all outgoing edges from $\mathcal{X}$ on a causal path from $\mathcal{X}$ to $Y$. This owes to the fact that $\mathfrak{q}$ contains none of these edges. Assume that $\mathfrak{q}$ contains one of these edges. As the descendant path from $m$ to $Y$ does not intersect $\mathcal{X}$ and $m\notin\mathcal{X}$ this edge has to be in $\mathfrak{p}$.  As $\mathcal{X}\cap\mathcal{I}=\emptyset$, there has to be $k\in\{1,\ldots,n\}$ such that $v_k\in\mathcal{X}$ and $\mathfrak{p}$ contains the edge $v_k\to$ which lies on a causal path from $\mathcal{X}$ to $Y$. Take the smallest of such $k$. There are two possibilities: 
     \begin{itemize}
         \item[$(\alpha)$] This edge points towards $m$, meaning we have $v_k\to v_{k+1}$. As by \textit{(2)} $\mathcal{B}$ is not a descendant of $\mathcal{X}$, $v_{k+1}$ is a non-collider in $\mathfrak{p}$, meaning in $\mathfrak{p}$ we find the structure $v_k\to v_{k+1}\to v_{k+2}$. Repeating this argument a large but finite number of times, we conclude that in $\mathfrak{p}$ we find the directed path $v_k\to v_{k+1}\to\cdots\to v_n\to m$, which is a contradiction with $R\cap DE_{\mathcal{G}_S}(\mathcal{X})=\emptyset$.
         \item[$(\delta)$] This edge points towards $i$. By the same argument as before we find in $\mathfrak{p}$ the structure $i\leftarrow v_1\leftarrow\cdots\leftarrow v_k$. The causal path from $\mathcal{X}$ to $Y$ which starts with the edge $v_k\to v_{k-1}$ (assume the notation $v_0=i$) can behave in two ways w.r.t. $\mathfrak{p}$:
         \begin{itemize}
             \item[$\ast$] $Y\notin\{v_1,\ldots,v_{k-1}\}$. Then for the segment of $\mathfrak{p}$  $v_k\to v_{k-1}\to\cdots\to v_1\to i=v_0$  there exist $l\in\{0,\ldots,k-1\}$ such that we find the path $v_l\to\cdots\to Y$ in $\mathcal{G}_S$, and the path $v_k\to\cdots\to v_l\to\cdots\to Y$ is causal. Take the smallest of such $l$'s (this allows us to assume that the path $v_l\to\cdots\to Y$ does not intersect the former segment of $\mathfrak{p}$  except at $v_l$. Now, the segment $v_{l}\to\cdots\to v_1\to i$ (if it exists) does not contain any edges outgoing $\mathcal{X}$ on a causal path from $\mathcal{X}$ to $Y$, by the choice of $k$. Neither does the path $v_l\to\cdots\to Y$, as the path $v_k\to\cdots\to v_l\to\cdots\to Y$ is causal. Therefore, the path $i\leftarrow\cdots\leftarrow v_l\to\cdots\to Y$ does not contain any edges outgoing $\mathcal{X}$ on a causal path from $\mathcal{X}$ to $Y$, and it does not intersect $\mathcal{B}$ by assumption \textit{(2)}. Thus, this path violates assumption \textit{(1)}. 
             \item[$\ast$]  $Y\in\{v_1,\ldots,v_{k-1}\}$. Then in $\mathfrak{p}$ we find the structure $v_k\to\cdots\to Y\to\cdots\to i$. The sub-path $Y\to\cdots\to i$ does not contain any edges outgoing $\mathcal{X}$ on a causal path from $\mathcal{X}$ to $Y$, as otherwise $\mathcal{G}_S$ would have a cycle. Furthermore, it contains only non-colliders not in $\mathcal{B}$, violating assumption \textit{(1)}.
         \end{itemize}
     \end{itemize}
     We can then conclude that $\mathfrak{q}$ is open given $\mathcal{B}$ and contains no edges outgoing $\mathcal{X}$ on a causal path from $\mathcal{X}$ to $Y$. This again constitutes a contradiction with assumption \textit{(1)}, what proves (a).

     Next we show that (b) also holds, namely that $\mathbb{E}[\mathbb{C}\text{ov}(\varepsilon,\mathcal{I}|\mathcal{B})]=0$. Recall  that $\varepsilon$ is a linear combination of the innovations which enter the structural equation of $Y$ and the iterated ancestors through the MA part of the process. Let $\varepsilon^i$ be one of these innovations which appear in $\varepsilon$. Then there is $A^i$ an ancestor of $Y$ (maybe $Y$ itself) such that $\varepsilon^i$ enters its structural equation directly, through its MA part. We will show that $\varepsilon^i$ does not appear in the MA($\infty$) representations of $\mathcal{B}$ and $\mathcal{I}$. By Proposition \ref{prop:DisjointMArepDirectVarmaWith} it suffices to show that  $SP_{\mathcal{G}_S}(A^i)\cap AN_{\mathcal{G}_S}(\mathcal{B})=\emptyset$ and $SP_{\mathcal{G}_S}(A^i)\cap AN_{\mathcal{G}_S}(\mathcal{I})=\emptyset$.
     \begin{itemize}
         \item For $\mathcal{B}$, recall that by construction there is a directed path from $\mathcal{X}$ to $Y$  going through $A^i$ in the full-time DAG, and hence also in $\mathcal{G}_S$ (if $A^i$ were not a descendant of $\mathcal{X}$  it would not be iterated over), which only intersects $\mathcal{X}\cup\mathcal{B}$ at the initial node (therefore this path is causal). Thus, $A^i$ is a descendant of $\mathcal{X}$, and by \textit{(2)} we conclude that  $SP_{\mathcal{G}_S}(A^i)\cap AN_{\mathcal{G}_S}(\mathcal{B})=\emptyset$.
         \item For $\mathcal{I}$, assume the opposite: there exists $\omega\in SP_{\mathcal{G}_S}(A^i)\cap AN_{\mathcal{G}_S}(\mathcal{I})$. Then we can consider the path $\mathfrak{r}: \mathcal{I}\leftarrow\cdots\leftarrow\omega\leftrightarrow A^i \to \cdots\to Y$. We know  that no node in the descendant path $A^i \to \cdots\to Y$ is in $\mathcal{B}$. Furthermore, no node in the descendant path $\omega\to\cdots\to\mathcal{I}$ intersects $\mathcal{B}$, as otherwise $\omega\in SP_{\mathcal{G}_S}(A^i)\cap AN_{\mathcal{G}_S}(\mathcal{B})$ which we know is an empty intersection. Thus, $\mathfrak{r}$ is an open path in $\mathcal{G}_S$ given $\mathcal{B}$. Furthermore, $\mathfrak{r}$ contains no edges outgoing $\mathcal{X}$ on a causal path from $\mathcal{X}$ to $Y$. Indeed, by construction $A^i \to \cdots\to Y$ does not intersect $\mathcal{X}$. For the side $\omega\to\cdots\to\mathcal{I}$ we reason analogously to item $(\delta)$ before, concluding that this side neither contains an edge outgoing $\mathcal{X}$ on a causal path from $\mathcal{X}$ to $Y$. Hence $\mathfrak{r}$ violates assumption \textit{(1)}, meaning that $SP_{\mathcal{G}_S}(A^i)\cap AN_{\mathcal{G}_S}(\mathcal{I})=\emptyset$, as we wanted to show.
     \end{itemize}
     This means that $\varepsilon^i$ does not enter the MA($\infty$) representation of $\mathcal{B}$ or $\mathcal{I}$, hence $\varepsilon^i\indep(\mathcal{B,I})$. As this holds for all the $\varepsilon^i$'s which enter $\varepsilon$ as a linear combination, $\varepsilon\indep(\mathcal{B,I})$, and therefore $\mathbb{E}[\mathbb{C}\text{ov}(\varepsilon,\mathcal{I}|\mathcal{B})]=0$, as we wanted to show. This concludes the proof of \textit{(i)}.

     The proofs of \textit{(ii)} and \textit{(iii)} follow almost identically the arguments given in \cite[Thm. 5]{thams2022identifying}. Part \textit{(ii)} is due to the fact that if $\mathbb{E}[\mathbb{C}\text{ov}(\mathcal{X},\mathcal{I}|\mathcal{B})]$ has full row rank, then if a solution to the moment equation $\mathbb{E}[\mathbb{C}\text{ov}(Y,\mathcal{I}|\mathcal{B})]=\beta\mathbb{E}[\mathbb{C}\text{ov}(\mathcal{X},\mathcal{I}|\mathcal{B})]$ exists, it is unique.

     For statement \textit{(iii)}, if $\mathcal{X,I,B}$ and $Y$ are assumed to have zero mean then we can center \textbf{X, I, B, Y} such that \cite[Eq. (5)]{thams2022identifying}: $$\hat{\beta}=\widehat{\mathbb{E}}[r_{\textbf{Y}}r_{\textbf{I}}^{\top}]W\widehat{\mathbb{E}}[r_{\textbf{I}}r_{\textbf{X}}^{\top}]\left(\widehat{\mathbb{E}}[r_{\textbf{X}}r_{\textbf{I}}^{\top}]W\widehat{\mathbb{E}}[r_{\textbf{I}}r_{\textbf{X}}^{\top}]\right)^{-1},$$ where $\widehat{\mathbb{E}}$ denotes the empirical mean, and $r_{\textbf{Y}}\defeq \textbf{Y}-\widehat{\mathbb{E}}[\textbf{Y}|\textbf{B}]$ are the residuals after regressing \textbf{Y} on \textbf{B} (\'idem for \textbf{I} and \textbf{X}). We want to show that $\hat{\beta}\stackrel{\mathbb{P}}{\to}\beta$. Firstly, note that the Law of large numbers holds for covariance-stationary processes (see \cite[Prop. 7.5]{hamilton2020time}). Thus, in combination with Slutsky's Theorem, we know that: $$\hat{\beta}\stackrel{\mathbb{P}}{\to} {\mathbb{E}}[r_{{Y}}r_{\mathcal{I}}^{\top}]W {\mathbb{E}}[r_{\mathcal{I}}r_{\mathcal{X}}^{\top}]\left({\mathbb{E}}[r_{\mathcal{X}}r_{\mathcal{I}}^{\top}]W {\mathbb{E}}[r_{\mathcal{I}}r_{\mathcal{X}}^{\top}]\right)^{-1},$$ where $r_Y=Y-\mathbb{E}[Y|\mathcal{B}]$ (\'idem for $\mathcal{X,I}$). Now  we can rewrite: $${\mathbb{E}}[r_{{Y}}r_{\mathcal{I}}^{\top}]={\mathbb{E}}[(r_{{Y}}-\beta r_{\mathcal{X}})r_{\mathcal{I}}^{\top}]+\beta{\mathbb{E}}[r_{\mathcal{X}}r_{\mathcal{I}}^{\top}]=\beta{\mathbb{E}}[r_{\mathcal{X}}r_{\mathcal{I}}^{\top}],$$ as the first addend equals zero due to the conditional uncorrelation proven in \textit{(i)}. Therefore, $\hat{\beta}\stackrel{\mathbb{P}}{\to}\beta$, and this concludes the proof.
\end{proof}

\end{document}